\documentclass[12pt]{amsart}
\usepackage{microtype}
\usepackage{etex} 
\usepackage[active]{srcltx}
\usepackage{calc,amssymb,amsthm,amsmath,amscd, eucal,ulem}
\usepackage{alltt}
\usepackage{mathtools}
\usepackage{verbatim}
\usepackage{mathtools}
\usepackage{enumerate}
\usepackage{bbding}
\RequirePackage[dvipsnames,usenames]{color}

\input{xy}
\xyoption{all}
\usepackage{amsfonts, mathrsfs}
\usepackage[cal=cm, calscaled=1.05]{mathalfa}
\usepackage{dcpic, pictexwd}
\usepackage[left=1in,top=1in,right=1in,bottom=1in]{geometry}
\usepackage{bm}
\usepackage{verbatim}
\usepackage{upgreek}
\usepackage{systeme}

\usepackage[all]{xy}
\SelectTips{cm}{10}
\usepackage[pdftex]{hyperref}

\DeclareMathOperator{\Hom}{Hom}

\DeclareMathOperator{\id}{id}
\DeclareMathOperator{\coker}{coker}
\DeclareMathOperator{\colim}{colim}
\DeclareMathOperator{\End}{End}

\DeclareMathOperator{\Spec}{Spec}

\DeclareMathOperator{\fpt}{fpt}
\DeclareMathOperator{\lct}{lct}

\newcommand{\eps}{\varepsilon}

\title{The associated graded module of the test module filtration}

\begin{document}

\swapnumbers
\theoremstyle{plain}
\newtheorem{Le}{Lemma}[section]
\newtheorem{Ko}[Le]{Corollary}
\newtheorem{Theo}[Le]{Theorem}
\newtheorem*{TheoB}{Theorem}
\newtheorem{Prop}[Le]{Proposition}
\newtheorem*{PropB}{Proposition}
\newtheorem{Con}[Le]{Conjecture}
\newtheorem{claim}[Le]{Claim}
\newtheorem{Assu}[Le]{Assumption}
\theoremstyle{definition}
\newtheorem{Def}[Le]{Definition}
\newtheorem*{DefB}{Definition}
\newtheorem{Bem}[Le]{Remark}
\newtheorem{Bsp}[Le]{Example}
\newtheorem*{BspB}{Example}
\newtheorem{Be}[Le]{Observation}
\newtheorem{Sit}[Le]{Situation}
\newtheorem{Que}[Le]{Question}
\newtheorem{Dis}[Le]{Discussion}
\newtheorem{Prob}[Le]{Problem}
\newtheorem{Konv}[Le]{Convention}
\newtheorem{Nota}[Le]{Notation}

\author{Axel St\"abler}

\address{Axel St\"abler\\
Johannes Gutenberg-Universit\"at Mainz\\ Fachbereich 08\\
Staudingerweg 9\\
55099 Mainz\\
Germany}

\email{staebler@uni-mainz.de}

\date{\today}


\begin{abstract}
We show that each direct summand of the associated graded module of the test module filtration $\tau(M, f^\lambda)_{\lambda \geq 0}$ admits a natural Cartier structure. If $\lambda$ is an $F$-jumping number, then this Cartier structure is nilpotent on $\tau(M, f^{\lambda -\eps})/\tau(M, f^\lambda)$ if and only if the denominator of $\lambda$ is divisible by $p$. We also show that these Cartier structures coincide with certain Cartier structures that are obtained by considering certain $\mathcal{D}$-modules associated to $M$ that were used to construct Bernstein-Sato polynomials.

Moreover, we point out that the zeros of the Bernstein-Sato polynomial $b_{M,f}$ attached to an \emph{$F$-regular} Cartier module correspond to its $F$-jumping numbers. This generalizes \cite[Theorem 5.4]{blicklestaeblerbernsteinsatocartier} where a stronger version of $F$-regularity was used. Finally, we develop a basic theory of \emph{non-$F$-pure modules} and prove a weaker connection between Bernstein-Sato polynomials $b_{M,f}$ and Cartier modules $(M, \kappa)$ for which $M_f$ is $F$-regular and certain jumping numbers attached to $M$.
\end{abstract}

\maketitle

\section*{Introduction}
For the purpose of this introduction let us consider a hypersurface $f$ inside a polynomial ring $R = \mathbb{F}_p[x_1, \ldots, x_n]$. We will often only state special cases of our main result and refer the reader to the paper for sharper statements. A classical theme is the study of singularities of $f$ via the Frobenius morphism $F$. Note that $F_\ast R$ is a free $R$-module with basis given by the monomials $x_1^{i_1} \cdots x_n^{i_n}$ with $0 \leq i_j \leq p-1$. We define an $R$-linear map $\kappa\colon F_\ast R \to R$ by sending the basis element $x_1^{p-1} \cdots x_n^{p-1}$ to $1$ and all other basis elements to zero. Given any rational $\lambda \geq 0$ one defines the test ideal \[\tau(R, f^\lambda) = \kappa^e (R \cdot f^{\lceil \lambda p^e\rceil})\] for any $e \gg 0$. In particular, one easily sees that $\tau(R, f^0) = R$. Varying $\lambda$ we obtain a non-increasing right-continuous filtration of ideals. The smallest $\lambda$ for which $\tau(R, f^\lambda) \neq R$ is called the $F$-pure threshold, $\fpt(f)$ for short. Similarly, we call any $\lambda$ such that $\tau(R, f^\lambda) \neq \tau(R, f^{\lambda -\eps})$ for all $0 < \eps \ll 1$ an $F$-jumping number.

If $f \in \mathbb{Q}[x_1,\ldots, x_n]$ and one considers the various reductions $f_p$ of $f$ to positive prime characteristic, then $\fpt(f_p) \xrightarrow{p \to \infty} \lct(f)$ and $\fpt(f_p) \leq \lct(f)$ for almost all $p$ (\cite[Theorem 6.8]{harayoshidagentightclosuremultideals}). Here $\lct(f)$ is the so-called log-canonical threshold which is a similar characteristic zero invariant that is defined using an embedded resolution of singularities. It is a deep and presently wide open conjecture that $\fpt(f_p) = \lct(f)$ for infinitely many $p$ (\cite[Conjecture 3.6]{mustatatakagiwatanabefthresholdsbernsteinsato}).

It has been observed for quasi-homogeneous hypersurfaces that if the log canonical threshold does not coincide with the $F$-pure threshold, then the denominator of the $F$-pure threshold is a $p$th power (see \cite[Theorem 3.5]{hbwzfthresholdshomogeneouspoly} or \cite[Lemma 3.7 (2)]{muellerfptquasihomo}). On the other hand, there are only two known (families of) examples where the $F$-pure threshold does not coincide with the log canonical threshold but the denominator of the $F$-pure threshold is not divisible by $p$ (see \cite[Example 4.5]{mustatatakagiwatanabefthresholdsbernsteinsato} and \cite[Proposition 2.7, Corollary 2.10]{cantonhswbehaviorofsingsatfpt}).

From the point of view of birational geometry $F$-jumping numbers with a denominator divisible by $p$ are special in the sense that the correspondence between certain Cartier linear maps and certain $\mathbb{Q}$-divisors breaks down in this case (e.g.\ \cite{schwedefadjunction}). This correspondence is central to many applications of test ideals in birational geometry.

\subsection*{Nilpotent Cartier structures}
One goal of this article is to further illustrate that $F$-pure thresholds whose denominators are divisible by $p$ are special in the following more general context: A finitely generated $R$-module $M$ endowed with an $R$-linear map $\kappa\colon F_\ast M \to M$ is called a Cartier module (cf.\ \cite{blicklep-etestideale}). One can associate to $M$ and $f \in R$ a test module filtration $\tau(M, f^\lambda)_{\lambda \geq 0}$ that has similar properties as in the case $M = R$. In particular, we can form the associated graded module $\bigoplus_{\lambda > 0} Gr^\lambda M =  \bigoplus_{\lambda > 0} \tau(M, f^{\lambda-\eps})/\tau(M, f^\lambda)$. In Section \ref{SectionNilpotenceAssGraded} we will attach a natural Cartier module structure to each direct summand and show that the Cartier module is \emph{nilpotent} if and only if the denominator of $\lambda$ is divisible by $p$. Here nilpotent means that some power of the structural map $\kappa$ acts as zero on the module. 

This notion of nilpotence is interesting for the following reason. Nilpotent Cartier modules form a Serre subcategory so that we may consider the attached localized category of Cartier crystals. Roughly speaking, this means killing all nilpotent Cartier modules and working up to nilpotence. This category is then equivalent (\cite{blickleboecklecartierfiniteness}) to the category of unit $R[F]$-modules of Emerton and Kisin (\cite{emertonkisinrhunitfcrys}) and anti-equivalent to the category of perverse constructible $\mathbb{F}_p$-sheaves on the \'etale site associated to $\Spec R$.

The first main result that we obtain is

\begin{TheoB}[Theorem \ref{TheoNilpotentGr}]
Let $R$ be essentially of finite type over an $F$-finite field, $(M, \kappa)$ a Cartier module, $f \in R$ and $\lambda$ an $F$-jumping number of the test module filtration of $M$ along $f$. Then for an integer $a$ the map $\kappa^e f^{a}$ defines a Cartier structure on $Gr^\lambda M$ if and only if $a \geq \lceil \lambda (p^e -1) \rceil$. Moreover, in this case the Cartier structure is not nilpotent if and only if $a = \lambda(p^e-1)$ and this quantity is an integer. In particular, if the denominator of $\lambda$ is divisible by $p$, then all these Cartier structures are nilpotent.
\end{TheoB}

Secondly, we also show, extending and vastly simplifying some results of \cite[Section 4]{blicklestaeblerbernsteinsatocartier}, that test modules admit a simple description, akin to the case of the test ideal in a polynomial ring, in many cases. A special instance is the following

\begin{TheoB}[cf.\ Theorem \ref{TauFRegularDescription}]
Let $R$ be essentially of finite type over an $F$-finite field, $(M, \kappa)$ an $F$-regular Cartier module  and $f \in R$ an $M$-regular element, then one has $\tau(M, f^\lambda) = \kappa^e f^{\lceil \lambda p^e \rceil} M$ for all $e \gg 0$.
\end{TheoB}

\subsection*{Nilpotence and its relation to Bernstein-Sato polynomials}
We will also relate the Cartier structures in Theorem \ref{TheoNilpotentGr} to certain Cartier structures obtained from generalized eigenspaces of higher Euler operators that play a crucial role when constructing Bernstein-Sato polynomials in positive characteristic (see \cite{blicklestaeblerbernsteinsatocartier}). Similar results were observed by Bitoun (\cite{bitounbernsteinsatoposchar}) for the case that $M = R$ in the (equivalent) framework of unit $R[F]$-modules. His proof relies on formal properties of $p$-adic expansions. We will show that it is in fact also a formal consequence of the machinery of test modules via Theorem \ref{TheoNilpotentGr}. Let us explain this in more detail.

Let $\mathcal{D}^e_R \cong \End_R(F_\ast^e R)$ be the subring of the ring of differential operators $\mathcal{D}_R$ consisting of operators which are linear over $R^{p^e}$. Adding a variable and considering the ring of differential operators $\mathcal{D}_{R[t]}$ we have the so-called divided power Euler operators $\theta_{p^i}$ which act on a monomial $rt^m$ as $\binom{m}{p^i} rt^m$. Any $\mathcal{D}^e_R[\theta_{p^0}, \theta_p, \ldots, \theta_{p^{e-1}}]$-module admits a direct sum decomposition into generalized eigenspaces $E_i$, with $i \in \mathbb{F}_p^e$, where $\theta_{p^{j-1}}$ acts on $E_i$ by multiplication with $i_j$.

Let now $(M, \kappa)$ be a Cartier module, $f \in R$ an $M$-regular element and consider the graph embedding \[\gamma\colon \Spec R \longrightarrow \Spec R[t], t \longmapsto f. \] 
By adjunction for finite maps a Cartier structure $\kappa^e\colon F_\ast^e M \to M$ is equivalent to a map $C^e\colon M \to \Hom_R(F^e_\ast R, M) \eqqcolon {F^e}^! M$ and there is a natural map $c\colon \gamma_\ast {F^e}^! M \to {F^{e}}^! \gamma_\ast M$. The latter module naturally is a $\mathcal{D}^e_{R[t]}$-module by precomposition.

Then we can, for each $e$, consider the quotient \[N_e \coloneqq c(\gamma_\ast C^e(M))\mathcal{D}^e_R[\theta_1, \theta_p, \ldots, \theta_{p^{e-1}}]/c(\gamma_\ast C^e(fM))\mathcal{D}^e_R[\theta_1, \theta_p, \ldots, \theta_{p^{e-1}}]\] and its decomposition into generalized eigenspaces $E_i$ with respect to the $\theta$ as above.

Lifting the $i=(i_1, \ldots,i_e)$ for which the generalized eigenspaces of $N_e$ are non-zero to $\mathbb{Z}$ via $\mathbb{F}_p = \{0, \ldots, p-1\}$ we set $x_i \coloneqq (i_1 + i_2p + \cdots + i_e p^{e-1})p^{-e}$ and define the $e$th Bernstein-Sato polynomial \[b^e_{M,f}(s) \coloneqq \prod_{E_i \neq 0} (s-x_i) \in \mathbb{Q}[s].\]

\begin{TheoB}[cf.\ Theorem \ref{BSPGeneral1} and Corollary \ref{BernsteinSatoExtension}]
Let $R$ be regular and essentially of finite type over an $F$-finite field, $(M,\kappa)$ an $F$-regular Cartier module and $f$ and $M$-regular element. Then for all $e \gg 0$ the roots of $b_{M,f}^e(s)$ are given by $\frac{\lceil \lambda p^e \rceil - 1}{p^e}$ where $\lambda \in (0,1]$ runs over the $F$-jumping numbers of $\tau(M, f^\lambda)$.
\end{TheoB}

Here $F$-regularity is in the sense of \cite{blicklestaeblerfunctorialtestmodules}. This result strengthens \cite[Theorem 5.4]{blicklestaeblerbernsteinsatocartier} where the stronger version of $F$-regularity as defined in \cite{blicklep-etestideale} was used.

One has $\bigcup_{e \geq 0} \mathcal{D}^e_{R[t]} = \mathcal{D}_{R[t]}$ and the $N_e$ form a direct system so that one can also consider $\colim_e N_e$ and then look at the simultaneous generalized eigenspace for all Euler operators $\theta_{p^i}$. As it turns out this approach only recovers those $F$-jumping numbers $\lambda$ that lie in $(0,1] \cap \mathbb{Z}_{(p)}$.

\begin{TheoB}[cf.\ Theorem \ref{AnotherMainResult}]
Under the assumptions as in the previous theorem, assume further that $\lambda \in \mathbb{Q}$ is an $F$-jumping number of the test module filtration of $M$ along $f$. In particular, $\frac{\lceil \lambda p^e\rceil -1}{p^e}$ corresponds for all $e \gg 0$ to a zero $x_i$ of $b^e_{M,f}(s)$. Then there is an isomorphism of $\mathcal{D}^e_R[\theta_1, \ldots, \theta_{p^{e-1}}]$-modules ${F^e}^! Gr^\lambda M \to E_i$, where $E_i$ is the generalized $i$-eigenspace associated to $N_e$. This isomorphism induces a transition map \[{F^e}^! Gr^\lambda M \to {F^{e+a}}^! Gr^\lambda M.\] This transition map is induced by a non-nilpotent morphism \[C\colon Gr^\lambda M \to {F^s}^! Gr^\lambda M\] if and only if $\lambda(p^s -1) \in \mathbb{Z}$ and $e,a$ are multiples of $s$. Moreover, in this case $C$ is the adjoint of $\kappa^s f^{\lambda(p^s -1)}$.
\end{TheoB}

\subsection*{Non-$F$-pure modules and relation to Bernstein-Sato polynomials}
In Section \ref{Nonfpuremodules} we study basic properties of so-called non-$F$-pure modules $\sigma(M, f^\lambda)$. This section is, apart from some basic results in Section \ref{SectionNilpotenceAssGraded}, largely independent from the preceding sections. The $\sigma(M, f^\lambda)$ are a generalization of non-$F$-pure ideals studied in \cite{fujinotakagischwedenonlc}. We will see that some of the pathologies that non-$F$-pure ideals exhibit in comparison to their characteristic zero analogues, namely \emph{non-lc ideals}, are constrained to the cases where the ``jumps'' have a denominator divisible by $p$.

\begin{TheoB}[cf.\ Proposition \ref{SigmaEqualsTauFregularLocus}, Corollary \ref{Sigmataupdenominator}]
Let $R$ be essentially of finite type over an $F$-finite field, $(M, \kappa)$ a Cartier module and $f$ and $M$-regular element. Assume that $M_f$ is $F$-regular. Let $\lambda > 0$ be a rational number. Then the following hold:
\begin{enumerate}[(a)]
\item{If the denominator of $\lambda$ is not divisible by $p$, then $\sigma(M, f^\lambda) = \tau(M, f^{\lambda - \eps})$ for all $0 < \eps \ll 1$.}
\item{If the denominator of $\lambda$ is divisible by $p$, then $\sigma(M, f^\lambda) = \tau(M, f^{\lambda + \eps})$  for all $0 < \eps \ll 1$.}
\end{enumerate}
\end{TheoB}

We will obtain a partial generalization of the correspondence between $F$-jumping numbers and zeros of Bernstein-Sato polynomials for the case that $(M, \kappa)$ is $F$-regular to Cartier modules $(M,\kappa)$ for which $M_f$ is $F$-regular. Let us write $Gr_\sigma^\lambda M = \sigma(M, f^\lambda)/\sigma(M, f^{\lambda+\eps})$. With this notation we prove in Section \ref{BspSigmaSection}:

\begin{TheoB}[cf.\ Theorem \ref{LastTheo}]
Let $R$ be regular and essentially of finite type over an $F$-finite field. Let $(M, \kappa)$ be a Cartier module, $f \in R$ an $M$-regular element and assume that $M_f$ is $F$-regular. Let $\lambda \in (0,1] \cap \mathbb{Z}_{(p)}$. If $\frac{\lceil\lambda p^e\rceil -1}{p^e}$ is a zero of the $e$th Bernstein-Sato polynomial for some $e \gg 0$ such that $\lambda(p^e -1) \in \mathbb{Z}$, then $Gr_\sigma^\lambda M \neq 0$.
\end{TheoB}

We note that it is also possible to formulate the preceding theorem entirely in the language of test modules. It is however the hope of the author that this form generalizes to arbitrary Cartier modules $(M, \kappa)$.

\subsection*{Acknowledgements} I thank Manuel Blickle, Mircea Musta\c{t}\u{a} and Kevin Tucker for useful discussions. Part of this paper was conceived while the author was visiting the University of Utah. I thank Karl Schwede for inviting me and for inspiring discussions. I am also indebted to an anonymous referee who pointed out a mistake in a previous version and whose comments helped to improve the exposition. The author acknowledges support by grant STA 1478/1-1 and by SFB/Transregio 45 Bonn-Essen-Mainz of the Deutsche Forschungsgemeinschaft (DFG). 
\section{A brief review of test modules}
\label{TestmoduleReview}

In this section we review very briefly the necessary facts on test modules that we need. We refer the reader to \cite{blicklestaeblerfunctorialtestmodules} for a detailed treatment.

Fix an $F$-finite ring $R$. A \emph{Cartier module} is a pair $(M, \kappa)$, where $M$ is an $R$-module and $\kappa\colon F_\ast^e M \to M$ is an $R$-linear map. We will always assume that $e =1$ except when we consider certain subquotients $Gr^\lambda M$ and $Gr_\sigma^\lambda M$ in Sections \ref{SectionNilpotenceAssGraded} and \ref{BspSigmaSection}. If we assume that $\Spec R$ is embeddable into a smooth scheme, then for given $e$ there is a contravariant functor to the category constructible $\mathbb{F}_{p^e}$-sheaves on the \'etale site. If we localize at nilpotent Cartier modules (to be defined below), call the resulting category \emph{Cartier crystals}, then this functor induces an anti-equivalence between Cartier crystals and perverse constructible $\mathbb{F}_{p^e}$-sheaves on the \'etale site (see \cite{schedlmeiercartierpervers} and references therein).

Given a Cartier module $(M, \kappa)$ and $f \in R$ and $\lambda$ a non-negative rational number we can form $R$-linear maps $\kappa^e f^{\lceil \lambda p^e \rceil}\colon F_\ast^e M \to M$ given by $m \mapsto \kappa^e(f^{\lceil \lambda p^e \rceil} m)$ for varying $e \geq 1$. The collection of these maps with addition induced by the one in $M$ and multiplication by composition form an $\mathbb{N}$-graded subring $\mathcal{C}$ of $\bigoplus_{e \geq 0} \Hom(F_\ast^e M, M)$, where we set $\mathcal{C}_0 = R$. It has both a left and a right $R$-module structure that are related via $r \varphi = \varphi r^{p^e}$ for any homogeneous element $\varphi$ of degree $e$. This ring is a special case of a so-called \emph{Cartier algebra}.

\begin{Def}
Let $R$ be an $F$-finite ring. A \emph{Cartier algebra} is an $\mathbb{N}$-graded ring $\bigoplus_{e \geq 0} \mathcal{C}_e$ with $\mathcal{C}_0 = R$ satisfying the relation $r \varphi = \varphi r^{p^e}$ for any $\varphi \in \mathcal{C}_e$ and $r \in R$.
\end{Def}

As usual $\mathcal{C}_+ = \bigoplus_{e \geq 1} \mathcal{C}_e$. We will write $\mathcal{C}_+^h$ for $(\mathcal{C}_+)^h$.

\begin{Bem}
\label{RoundedCartierAlgebras}
We will mostly use Cartier algebras of the form $\kappa^e f^{\lceil \lambda p^e \rceil}$ in this article. The reader familiar with test ideals may notice that people also often use algebras of the form $\kappa^e f^{\lceil \lambda (p^e -1) \rceil}$. If one computes test modules then both notions yield the same result (cf.\ \cite[Lemma 3.1]{staeblertestmodulnvilftrierung}). However, the categories of crystals are not the same and this will play an important role later on (cf.\ Section \ref{Nonfpuremodules}).
\end{Bem}

A $\mathcal{C}$-module means a left module over $\mathcal{C}$. We will moreover always assume that it is finitely generated as an $R$-module. We call a $\mathcal{C}$-module $M$ nilpotent if $\mathcal{C}_+^a M = 0$ for some (equivalently all) $a \gg 0$. A morphism $\varphi\colon M \to N$ of $\mathcal{C}$-modules is a \emph{nil-isomorphism} if its kernel and cokernel are nilpotent.

If $M$ is a $\mathcal{C}$-module then the descending chain $\mathcal{C}_+ M \supseteq \mathcal{C}_+^2 M \supseteq \ldots$ stabilizes (see \cite[Proposition 2.13]{blicklep-etestideale}) and we denote its stable member by $\underline{M}$.

\begin{Def}
The test module $\tau(M, f^\lambda)$ is the smallest $\mathcal{C}$-submodule $N$ of $M$ such that the inclusion $H^0_\eta(N)_\eta \subseteq H^0_\eta(M)_\eta$ is a nil-isomorphism for every associated prime $\eta$ of the $R$-module $M$.
\end{Def}

At this point we encourage the reader to take a look at \cite[Sections 1 and 2]{blicklestaeblerfunctorialtestmodules} for further discussion. It is proven in \cite[Theorems 3.4 and 3.6]{blicklestaeblerfunctorialtestmodules} that test modules exist if $R$ is essentially of finite type over an $F$-finite field. Moreover, in this case the test module filtration $\tau(M, f^\lambda)_{\lambda \geq 0}$ is a non-increasing right-continuous filtration. Under suitable finiteness conditions (e.g.\, $R$ is essentially of finite type over an $F$-finite field) this filtration is also discrete. Many other formal properties like Skoda also hold in this more general situation (see \cite[Section 4]{blicklestaeblerfunctorialtestmodules}). We call a number $\lambda$ such that $\tau(M, f^\lambda) \neq \tau(M, f^{\lambda - \eps})$ for all $0 < \eps \leq \lambda$ an \emph{$F$-jumping number}.

Finally, we say that a prime $\eta \in \Spec R$ of a $\mathcal{C}$-module $M$ is an associated prime of $M$ if $H^0_\eta(M)_\eta$ is not nilpotent. These form a subset of the associated primes of the underlying $R$-module.

We call a Cartier module $(M, \kappa)$ \emph{$F$-regular} if $\tau(M, f^0) = M$. The smallest $\lambda > 0$ such that $\tau(M, f^\lambda) \neq \tau(M, f^0)$ is called the \emph{$F$-pure threshold} of $f$ with respect to $M$. More generally, if we have a $\mathcal{C}$-module $N$, where $\mathcal{C}_e = \kappa^e f^{\lceil \lambda p^e \rceil}$, then we say that  $N$ is $F$-regular if $\tau(N, f^\lambda) = N$.

If $N$ is a $\mathcal{C}$-module and $\eta_1, \ldots, \eta_n$ are its associated primes, then we call $c_1, \ldots, c_n$ a \emph{sequence of test elements} if $c_i \notin \eta_i$ and the $\underline{H^0_{\eta_i}(N_{c_i})}$ are $F$-regular. If all associated primes of $M$ are minimal then we only need a single test element and this condition simplifies: If $N$ is a $\mathcal{C}$-module whose associated primes are minimal, then we call $c \in R$ a test element if $c$ is not contained in any minimal prime of $N$ and $\underline{N_c}$ is $F$-regular.

With this notion one has

\begin{Theo}
Let $R$ be essentially of finite type over an $F$-finite field, $(M, \kappa)$ a Cartier module and $f \in R$. Then there exists a sequence of test elements $c_1, \ldots, c_n$ for the $\mathcal{C}$-module $M$, where $\mathcal{C}_e = \kappa^e f^{\lceil \lambda p^e \rceil}$, and one has
\[\tau(M, f^\lambda) = \sum_{e \geq e_0} \sum_{i=1}^n \mathcal{C}_{e} c_i^{a_i} \underline{H^0_{\eta_i}(M)} \]
for any $e_0 \geq 0$ and any $a_i \geq 1$.
Moreover, if $M$ only has minimal associated primes, then this simplifies to
\[\tau(M, f^\lambda) = \sum_{e \geq e_0} \mathcal{C}_e c^a \underline{M}.\]
\end{Theo}
\begin{proof}
See \cite[Theorem 3.4, Theorem 3.6]{blicklestaeblerfunctorialtestmodules} for the general case and \cite[Theorem 3.11]{blicklep-etestideale} for the special case
\end{proof}

It is mostly this presentation that we will be used in this article. Also note that we will prove shortly that, if $(M, \kappa)$ is $F$-regular, then one has in fact $\tau(M, f^\lambda) = \kappa^e f^{\lceil \lambda p^e \rceil} M$ (Theorem \ref{TauFRegularDescription} below).

\section{Nilpotence of the direct summands of the associated graded module}
\label{SectionNilpotenceAssGraded}
Throughout this section $R$ is a ring essentially of finite type over an $F$-finite field. This assumption is imposed to ensure existence of test modules (see \cite[Theorem 3.6]{blicklestaeblerfunctorialtestmodules}; in our setup this automatically implies existence of a sequence of test elements -- cf.\ \cite[Remark 3.7]{blicklestaeblerfunctorialtestmodules}) and discreteness of the filtration. Granting these notions our arguments work for arbitrary $F$-finite rings.

Fix a Cartier module $(M, \kappa)$ and $f \in R$.
In \cite[Proposition 4.5]{staeblertestmodulnvilftrierung} the author defined a Cartier structure on the direct summands of the associated graded module of the test module filtration. Namely, if $Gr^\lambda(M) = \tau(M,f^{\lambda-\eps})/\tau(M, f^\lambda)$, then $\kappa f^{\lceil \lambda(p-1)\rceil}$ operates on this quotient. In fact, more generally $\kappa^e f^{\lceil \lambda(p^e-1)\rceil}$ operates on $Gr^\lambda(M)$ and one easily checks if $Gr^\lambda M$ is nilpotent with respect to this Cartier structure, then also with respect to the one above. The map $\kappa^e f^{\lceil \lambda(p^e -1) \rceil}$ operators on $Gr^\lambda(M)$ due to Skoda and due to the fact, as we shall point out in detail below, that $\kappa^e \tau(M, f^\lambda) = \tau(M, f^{\lambda/p^e})$.

While this definition may seem ad hoc we point out that if $i\colon \Spec R/(f) \to \Spec R$ denotes the natural closed immersion and $f$ is a non-zero-divisor on $R$ then for any Cartier module $(M, \kappa)$ on $\Spec R$ the induced Cartier structure on $R^1i^! M$ is given by $\kappa f^{p-1}$ (cf.\ \cite[Example 3.3.12]{blickleboecklecartiercrystals}). By Skoda one easily sees that the support of $Gr^\lambda(M)$ is contained in $\Spec R/(f)$. We will in fact see shortly that these Cartier structures are very natural.

The next lemma was already proven in \cite[Proposition 3.2]{staeblertestmodulnvilftrierung} for the case that $M$ has only minimal primes. We give a simplified proof here.

\begin{Le}
\label{CartierTauDivisionByP}
Let $(M, \kappa)$ be a Cartier module and $f \in R$. Then for all $\lambda \geq 0$ we have $\kappa(\tau(M, f^\lambda)) = \tau(M, f^{\frac{\lambda}{p}})$.
\end{Le}
\begin{proof}
By virtue of \cite[Theorem 3.4]{blicklestaeblerfunctorialtestmodules} we have \[\tau(M, f^\lambda) = \sum_{i=1}^n \sum_{e \geq e_0} \kappa^e f^{\lceil \lambda p^e\rceil} c_i \underline{H^0_{\eta_i}(M)}\] for any $e_0 \geq 0$, where the $\eta_i$ are the associated primes of $M$ and the $c_i$ form a sequence of test elements in the sense of \cite[Definition 3.1]{blicklestaeblerfunctorialtestmodules}. We thus have \begin{align*} \kappa (\tau(M,f^\lambda)) &= \sum_{i=1}^n \sum_{e \geq e_0} \kappa^{e+1} f^{\lceil \frac{\lambda}{p} p^{e+1}\rceil} c_i \underline{H^0_{\eta_i}(M)}  \\&= \sum_{i=1}^n \sum_{e \geq e_0 +1} \kappa^{e} f^{\lceil \frac{\lambda}{p} p^{e}\rceil} c_i \underline{H^0_{\eta_i}(M)} = \tau(M, f^{\frac{\lambda}{p}}). \end{align*}
\end{proof}

\begin{Le}
\label{Remarkable}
Let $(M, \kappa)$ be a Cartier module and $f \in R$. Assume that $\lambda$ is an $F$-jumping number of $\tau(M, f^\mu)$. For any integer $a$ the map given by $\kappa^e f^{a}$ induces a Cartier structure on $Gr^{\lambda} M$ if and only if $a \geq \lceil \lambda (p^e -1)\rceil$.
\end{Le}
\begin{proof}
Using Lemma \ref{CartierTauDivisionByP} and Skoda (\cite[Proposition 4.1]{blicklestaeblerfunctorialtestmodules}) one has \[\kappa^e f^{a} \tau(M, f^\lambda) = \tau(M, f^{\frac{\lambda + a}{p^e}}).\] For this to induce a Cartier structure on the quotient $Gr^\lambda M = \tau(M,f^\lambda)/\tau(M, f^{\lambda - \eps})$ we must have $\frac{\lambda + a}{p^e} \geq \lambda$. Equivalently, $a \geq \lambda(p^e  -1)$ and since $a$ is an integer this is equivalent to $a \geq \lceil \lambda (p^e -1)\rceil$. One similarly checks that in this case $\kappa^e f^{a} \tau(M, f^{\lambda - \eps}) \subseteq \tau(M, f^{\lambda - \eps})$ using that the test module filtration is non-increasing.
\end{proof}

The first main result of this section is

\begin{Theo}
\label{TheoNilpotentGr}
Let $(M, \kappa)$ be a Cartier module, $f \in R$ and $\lambda$ an $F$-jumping number of the test module filtration of $M$ along $f$. Let $a \geq \lceil \lambda (p^e-1)\rceil$. Then the Cartier structure on $Gr^\lambda M$ defined by $\kappa^e f^{a}$ is not nilpotent if and only if $a = \lambda(p^e-1)$ and this quantity is an integer. In particular, if the denominator of $\lambda$ is divisible by $p$, then all these Cartier structures are nilpotent.
\end{Theo}
\begin{proof}
We have just seen in Lemma \ref{Remarkable} that $\kappa^e f^{a}$ defines a Cartier structure on $Gr^\lambda M$.
We may write $\lceil \lambda(p^e-1) \rceil = \lambda(p^e -1) + \delta$ with $0 \leq \delta < 1$ and $\delta = 0$ if and only if $\lambda(p^e -1 )$ is an integer. By Skoda (\cite[Proposition 4.1]{blicklestaeblerfunctorialtestmodules}) we have \[\kappa^e f^{\lambda(p^e -1) + \delta} \tau(M, f^{\lambda -\eps}) = \kappa^e \tau(M, f^{\lambda- \eps + \lambda(p^e -1) + \delta}) = \kappa^e \tau(M,f^{\lambda p^e + \delta - \eps}).\] 
Now we use Lemma \ref{CartierTauDivisionByP} and obtain \[\kappa^e \tau(M, f^{\lambda p^e + \delta -\eps})  = \tau(M, f^{\lambda + \frac{\delta - \eps}{p^e}}).\] Note in fact, that we may take any $\eps'$ such that $\eps> \eps' >0 $ and still have $\tau(M,f^{\lambda -\eps'}) = \tau(M, f^{\lambda-\eps})$. Hence, if $\delta >0$ this actually forces $\eps < \delta$. But then $\tau(M,f^{\lambda + \frac{\delta - \eps}{p^e}}) \subseteq \tau(M, f^\lambda)$ which shows that the Cartier structure is nilpotent. The same argument also shows nilpotence for any $a > \lceil \lambda(p^e -1)\rceil$.
\end{proof}

\begin{Bem}
Note that with the notation of the proof of Theorem \ref{TheoNilpotentGr} if $\mu < \lambda$ is the previous $F$-jumping number then necessarily $\lambda- \mu \leq \delta$. In particular, if $\lambda$ is the $F$-pure threshold, then it actually follows that if $\delta \neq 0$ then $\delta \geq \lambda$ since we can take any $0 < \eps < \lambda$ without changing $\tau(M, f^{\lambda-\eps})$. Still assuming that $\lambda$ is the $F$-pure threshold and writing $\delta = \lceil \lambda(p^e -1) \rceil - \lambda(p^e -1) \geq \lambda$ we obtain equivalently that $\lceil \lambda(p^e -1) \rceil \geq \lambda p^e$. From this one easily deduces that $\lambda \notin (\frac{a}{p^e}, \frac{a}{p^e -1})$ for any integer $0 \leq a \leq p^e -1$. This is the analogue of \cite[Proposition 4.2]{hernandezfpurityhypersurfaces} for modules.
\end{Bem}

\begin{Bsp}
Consider the cusp $f = x^2 + y^3 \in \mathbb{F}_p[x,y] = R$. If $\lambda$ denotes the $F$-pure threshold then $\tau(R, f^\lambda) = (x,y)$ and for $p > 3$ one has (see \cite[Example 4.3]{mustatatakagiwatanabefthresholdsbernsteinsato}) \[\lambda =
\begin{cases}
\frac{5}{6} & p \equiv 1 \mod 3,\\
\frac{5}{6} - \frac{1}{6p} & p \equiv 2 \mod 3.
\end{cases}\] Moreover, $\lambda = \frac{1}{2}$ for $p = 2$ and $\lambda = \frac{2}{3}$ if $p =3$. So if $p \equiv 1 \mod 3$ we can take $e =1$ and obtain $\frac{5}{6} (p -1) \in \mathbb{Z}$. Since the quotient $Gr^\lambda$ is just $\mathbb{F}_p$ and we know that the obtained Cartier structure is not nilpotent it has to be $\kappa = \id$.

Of course, for any $p$ one has $Gr^\lambda = \mathbb{F}_p$ (as an $R$-module quotient) so that it admits the non-nilpotent Cartier structure $\kappa = \id$.
\end{Bsp}

We end this section by proving a simple description of $\tau(M, f^\lambda)$ in the case where $(M, \kappa)$ is $F$-regular. This is very useful for computations. We will use it in the next section to extend the relation of $F$-jumping numbers and zeros of Bernstein-Sato polynomials (\cite{blicklestaeblerbernsteinsatocartier}) to $F$-regular Cartier modules.

\begin{Theo}
\label{TauFRegularDescription}
Let $(M, \kappa)$ be a Cartier module and $f$ an $M$-regular element. If $(M, \kappa)$ is $F$-regular, then $\tau(M, f^\lambda) = \kappa^e f^{\lceil \lambda p^e \rceil} M$ for all $e \gg 0$.
\end{Theo}
\begin{proof}
By $F$-regularity we have $M = \tau(M, f^0)$. Hence, by Skoda (\cite[Proposition 4.1]{blicklestaeblerfunctorialtestmodules}) and Lemma \ref{CartierTauDivisionByP} we get \[\kappa^e f^{\lceil \lambda p^e \rceil} M = \kappa^e f^{\lceil \lambda p^e \rceil} \tau(M, f^0) = \kappa^e \tau(M, f^{\lceil \lambda p^e\rceil}) = \tau(M, f^{\lceil \lambda p^e \rceil p^{-e}}).\]
Since $\lambda \leq \frac{\lceil \lambda p^e\rceil}{p^e} \leq \lambda + \frac{1}{p^e}$ we conclude by right-continuity that $\tau(M, f^{\lceil \lambda p^e \rceil p^{-e}}) = \tau(M, f^{ \lambda})$ for all $e$ sufficiently large.
\end{proof}

\begin{Ko}
\label{TauStuff}
Let $(M, \kappa)$ be a Cartier module and $f$ an $M$-regular element. Then $\tau(M, f^\lambda) = \tau(\tau(M, f^0), f^\lambda)$. In particular, we have $\tau(M, f^\lambda) = \kappa^e f^{\lceil \lambda p^e \rceil} \tau(M, f^0)$.
\end{Ko}
\begin{proof}
Clearly, $\tau(M, f^0) \subseteq M$ so that $\tau(\tau(M, f^0), f^\lambda) \subseteq \tau(M, f^\lambda)$ by \cite[Proposition 1.15]{blicklestaeblerfunctorialtestmodules}.

For the other inclusion, by definition $\tau(M, f^0)$ is the smallest submodule of $M$ for which the inclusions $H^0_\eta(\tau(M, f^0))_\eta \subseteq H^0_\eta(M)_\eta$ are nil-isomorphisms with respect to $\kappa$ for all associated primes $\eta$ of $M$. Being a nil-isomorphism here just means that some power of $\kappa$ annihilates the cokernel. But then a fortiori some power of $\mathcal{C}_+ = \bigoplus_{e \geq 1} \kappa f^{\lceil \lambda p^e \rceil}$ acts as zero on this cokernel. By definition of $\tau(M, f^\lambda)$ this shows $\tau(M, f^\lambda) \subseteq \tau(M, f^0)$. Now we may apply $\tau(-, f^\lambda)$ to this inclusion to obtain $\tau(M, f^\lambda) \subseteq \tau(\tau(M, f^0), f^\lambda)$, where we again use \cite[Proposition 1.15]{blicklestaeblerfunctorialtestmodules} and the fact that $\tau$ (for a fixed Cartier algebra) is idempotent.

The final claim is an immediate application of Theorem \ref{TauFRegularDescription}.
\end{proof}

\begin{Bem}
The assumption on the $F$-regularity cannot be omitted, that is, if $M$ is not $F$-regular, then $\tau(M, f^\lambda) \neq \kappa^e f^{\lceil \lambda p^e\rceil} M$ in general. Consider for example $R = \mathbb{F}_p[x,y]$ and the Cartier module $M =\mathbb{F}_p[x,y]\cdot y^{-1} \subseteq j_\ast \mathbb{F}_p[x,y,y^{-1}]$, where $j\colon D(y) \to \Spec R$, with Cartier structure induced by localization.

Then $M$ is not $F$-regular since $\mathbb{F}_p[x,y]$ is a proper submodule that generically agrees with $M$. It is easy to see that $M$ is $F$-pure. Moreover, $y$ is a test element for $(M, \kappa)$ and therefore $\tau(M, x^\lambda) = x^{\lfloor \lambda \rfloor} R$ while $\kappa^e x^{\lceil \lambda p^e \rceil} M$ is not even contained in $R$ since $y^{-1}$ is a fixed point for the Cartier operation.

Similarly, if $M$ is not $F$-pure and has only minimal associated primes then $\tau(M, f^\lambda) \neq \kappa^e f^{\lceil \lambda p^e\rceil} \kappa^a c M$ in general, where $c$ is a test element. For example, consider the Cartier module $M = k[x] \cdot x^{-n} \subseteq k[x]$ with $n \geq 2$ and take $\lambda = 0$ and $f = x$. Clearly, $c = x$ is a test element and $\underline{M} = \kappa^a k[x] \cdot x^{-n} = k[x] \cdot x^{-1}$. We see that $\tau(M, x^0) = k[x]$ by the theorem. However, $\kappa^a x M = k[x] \cdot x^{-1}$ for all $a \gg 0$ so that (since $\lambda =0$) $\kappa^e x^{0} k[x] \cdot x^{-1} = k[x] \cdot x^{-1}$.
\end{Bem}

\section{Cartier structures on the direct summands of the associated graded module induced by differential operators}
\label{DModApproach}

Throughout this section we assume that $R$ is a regular ring essentially of finite type over an $F$-finite field. Regularity of $R$ is critical since we need that $F_\ast^e R$ is a flat $R$-module to ensure that ${F^e}^!$ is exact.

The goal of this section is to show that the Cartier structures defined on the direct summands of the associated graded module of $\tau(M, f^\lambda)$ at the beginning of Section \ref{SectionNilpotenceAssGraded} correspond to the Cartier structures obtained on quotients of generalized eigenspaces of certain $\mathcal{D}^e_R[\theta_1, \theta_p, \ldots, \theta_{p-1}]$-modules associated to the data $(M, f)$ for varying $e$. In particular, we will recover and generalize results of Bitoun (\cite{bitounbernsteinsatoposchar}).

We start by recalling the necessary $\mathcal{D}$-module theoretic notions (see \cite{mustatabernsteinsatopolynomialspositivechar} and \cite{blicklestaeblerbernsteinsatocartier} for more elaborate discussions). For any ring $R$ containing $\mathbb{F}_p$ the ring of $\mathbb{F}_p$-linear differential operators $\mathcal{D}_R \subseteq \End_{\mathbb{F}_p}(R)$ in the sense of Grothendieck (\cite[Definition 16.8.1]{EGAIV-4}) admits the so-called \emph{$p$-filtration} $\mathcal{D}_R = \bigcup_{e \geq 0} \mathcal{D}^e_R$, where $\mathcal{D}^e_R \cong \End_R(F^e_\ast R)$ (see \cite{chasepfiltration}).

\begin{Bsp}
If $R = k[t_1, \ldots, t_n]$ is a polynomial ring over a perfect field $k$, then $\mathcal{D}_R^e = R[\partial_{t_i}^{[1]}, \ldots, \partial_{t_i}^{[p^e -1]} \, \vert \, i =1, \ldots, n]$, where the $\partial_{t_i}^{[j]}$ are \emph{divided power operators} that act as follows \[\partial_{t_i}^{[j]} \bullet (t_1^{a_1} \cdots t_n^{a_n}) = \binom{a_i}{j} t_1^{a_1} \cdots t_{i-1}^{a_{i-1}} t_i^{a_i - j} t_{i+1}^{a_{i+1}} \cdots t_n^{a_n}.\] Moreover, one has $[\partial_{t_i}^{[a_i]}, \partial_{t_j}^{[a_j]}] = 0$ for $i \neq j$. 

As an aside we also mention that since $k$ is perfect $\mathcal{D}_R$ coincides with $k$-linear differential operators.
\end{Bsp}

\begin{Konv}
\label{konv.rightmodules}
Unless otherwise specified modules over rings of differential operators will always be \emph{right} modules.
\end{Konv}

Recall that since the Frobenius morphism is finite the upper shriek functor ${F^e}^!$ is given by  \[ {F^e}^!M \coloneqq \Hom_R(F_\ast^e R, M)\] for any $R$-module $M$. The module ${F^e}^! M$ naturally carries a right action of $\mathcal{D}^e_R \cong \End_R(F_\ast^e R)$ via \[{F^e}^! M \ni \varphi \longmapsto \varphi \circ f\] for any $f \in \End_R(F_\ast^e R)$.

\begin{Def}
\label{AdjointMap}
For any $e \geq 1$ there is a natural isomorphism \[\Sigma\colon \Hom_R(F^e_\ast M, M) \longrightarrow \Hom_R(M, {F^e}^! M), \quad \varphi \longmapsto [ m \mapsto (r \mapsto  \varphi(rm))].\] In this way, if $(M, \kappa)$ is a Cartier module, then the map $\kappa^e$ is equivalent to the map $\Sigma(\kappa^e)$. We will refer to the map $\Sigma(\kappa^e)$ as the \emph{adjoint} of $\kappa^e$ (and vice versa $\kappa^e$ is the adjoint of $\Sigma(\kappa^e)$.
\end{Def}

Let us now fix some assumptions and notations that shall be in force for the rest of this section:
\begin{Nota}
Recall that $R$ is a ring that is regular and essentially of finite type over an $F$-finite field $k$. We fix an element $f$ in $R$ and consider its so-called \emph{graph embedding}
\[ \gamma\colon \Spec R \longrightarrow \Spec R[t],\quad t \longmapsto f.\]

If $(M, \kappa)$ is a Cartier module, then $\gamma_\ast M$ also naturally carries a Cartier structure via \[F_\ast \gamma_\ast M \xrightarrow{\cong} \gamma_\ast F_\ast M \xrightarrow{\gamma_\ast \kappa} \gamma_\ast M\] where the first isomorphism is the natural isomorphism $F_\ast \gamma_\ast \cong \gamma_\ast F_\ast$. We will, by abuse of notation, denote this map again by $\kappa$.
\end{Nota}

\begin{Def}
Note that $\mathcal{D}^e_{R[t]}$ contains the differential operator $\theta_{a} \coloneqq t^{a} \partial_{t}^{[a]}$ for any $a < p^{e}$. In fact, $\mathcal{D}^e_{R[t]} = \mathcal{D}^e_R[t, \partial_t^{[a]} \, \vert \, a < p^e]$. We call the $\theta_a$ the \emph{(higher divided power) Euler operators}.
\end{Def}

\begin{Le}
\label{directedsystem}
Let $(M, \kappa)$ be a Cartier module on $R$ and $\gamma\colon \Spec R \to \Spec R[t], t \mapsto f$ the graph embedding. Then the adjoint of the Cartier structure ${F^{e}}^!\gamma_\ast M \to {F^{e+a}}^! \gamma_\ast M$ is given by \begin{align*}\Hom_{R[t]}(F_\ast^e R[t], \gamma_\ast M) &\longrightarrow \Hom_{R[t]}(F_\ast^{a+e} R[t], \gamma_\ast M) \\ \varphi &\longmapsto \kappa^a \circ F_\ast^a \varphi. \end{align*}
\end{Le}
\begin{proof}
First of all, note that duality of finite morphisms yields a map \[\gamma_\ast M \longrightarrow {F^a}^! \gamma_\ast M,\quad m \longmapsto [s \mapsto \kappa^a(\gamma^\#(s)m)],\] where $\gamma^{\#}\colon R[t] \to R, t \mapsto f$, corresponding to $F_\ast^a \gamma_\ast M \to \gamma_\ast M$. Applying ${F^e}^!$ (which is just a hom-functor) we get a map \[{F^{e}}^! \gamma_\ast M \longrightarrow {F^{e}}^! {F^a}^! \gamma_\ast M, \quad\varphi \longmapsto [r \mapsto [s \mapsto \kappa^a(\gamma^\#(r)(\varphi(s)))]].\] Now by tensor-hom adjunction we have ${F^e}^! {F^a}^! \cong {F^{e+a}}^!$ which sends a map as above to the map $[s \mapsto \kappa^a(\varphi(\gamma^\#(s)))]$.
\end{proof}

Quite generally, given any $\mathcal{D}^e_R[\theta_1, \theta_p, \ldots, \theta_{p^{e-1}}]$-module $A$ there is a unique decomposition
\[ A = \bigoplus_{i \in \mathbb{F}_p^e} E_i,\] where the $E_i$ are $\mathcal{D}_R$-modules and $\theta_{p^{j-1}}$ acts on $E_i$ by multiplication with $i_j$ for $j=1, \ldots, e$. We call this the decomposition into \emph{generalized eigenspaces}. This generalized eigenspace decomposition is preserved by morphisms. Moreover, it is compatible with the $p$-filtration in the following sense: $A$ is in particular a $\mathcal{D}^{e-1}_R[\theta_1, \theta_p, \ldots, \theta_{p^{e-2}}]$-module. As such it admits a decomposition \[ A = \bigoplus_{i \in \mathbb{F}_p^{e-1}} E_i\] into generalized eigenspaces. In this case, if we fix a tuple $i \in \mathbb{F}_p^{e-1}$, then \[E_i = \bigoplus_{j \in \mathbb{F}_p} E_{i,j},\] where on the right $E_{i,j}$ is the generalized eigenspace $(i_1, \ldots, i_{e-1}, j)$ obtained from the generalized eigenspace decomposition of $A$ considered as a $\mathcal{D}^e_R[\theta_1, \theta_p, \ldots, \theta_{p^{e-1}}]$-module. The proof of these facts are largely formal and rely on some properties of the $\theta_j$. We refer the reader to \cite[Lemma 3.1]{blicklestaeblerbernsteinsatocartier} or \cite[Proposition 4.2]{mustatabernsteinsatopolynomialspositivechar} for details.

Fix a Cartier module $(M, \kappa)$. Recall that ${F^e}^! \gamma_\ast M = \Hom_{R[t]}(F_\ast^e R[t], \gamma_\ast M)$ admits a $\mathcal{D}^e_{R[t]}$-module structure by premultiplication (see the discussion after \ref{konv.rightmodules}) and hence in particular a $\mathcal{D}^e_R[\theta_1, \theta_p, \ldots, \theta_{p^{e-1}}]$-module structure.
As such it admits a decomposition into generalized eigenspaces ${F^e}^! \gamma_\ast M = \bigoplus_{i \in \mathbb{F}_p^e} E_{i}$, where $\theta_{p^{j-1}}$ acts on $E_i$ by multiplication with $i_j$.

We now want to describe this generalized eigenspace structure in more detail. Given $t^m \in R[t]$ the operator $\theta_{p^a}$ acts via $\binom{m}{p^a} t^{m}$. By Lucas' Theorem  (see \cite[\S XXI]{lucasbinom}), if we write $m = i_1 + i_2 p + \cdots + i_e p^{e-1}$, then $\binom{m}{p^a}$ evaluates to $i_{a+1}$ (which, by convention, is zero if $a > e$).
Thus writing $i = (i_1, \ldots, i_e)$ we see that the projection onto the generalized eigenspace \[{F^e}^! \gamma_\ast M \longrightarrow E_i\] is induced by \[\pi\colon F_\ast^e R[t] \longrightarrow F^e_\ast R[t],\quad r t^m \longmapsto \begin{cases}r t^m, &  m = i_1 +i_2p + \ldots + i_e p^{e-1},\\ 0,& \text{else.}\end{cases}\]

Write $\gamma^{\#}$ for the map $R[t] \to R, t \mapsto f$ and note that  one has a natural map \[\gamma_\ast {F^e}^! M \longrightarrow {F^e}^! \gamma_\ast M, \quad \varphi \longmapsto \varphi \circ \gamma^{\#}.\] Recall (Definition \ref{AdjointMap}) that $\kappa^e\colon F_\ast^e M \to M$ corresponds to a map $C^e\colon M \to {F^e}^! M$ given by $[m \mapsto [r \mapsto \kappa^e(rm)]]$. By abuse of notation we will denote the image of  $C^e(M) = \{ [r \mapsto \kappa^e(rm) \vert m \in M\}$ under the natural map $\gamma_\ast {F^e}^!M \to {F^e}^! \gamma_\ast M$ by $\gamma_\ast C^e(M)$. This is an $R[t]$-submodule of the $\mathcal{D}_R^e[\theta_1, \theta_p, \ldots, \theta_{p^{e -1}}]$-module $F^e \gamma_\ast M$.

A key technical result (\cite[Corollary 5.3]{blicklestaeblerbernsteinsatocartier}) yields that for all $e \geq 1$ the generalized $(i_1, \ldots, i_e)$-eigenspace of the quotient \[N_e \coloneqq (\gamma_\ast C^e(M))\mathcal{D}^e_R[\theta_1, \theta_p, \ldots, \theta_{p^{e-1}}]/ (\gamma_\ast C^e(fM))\mathcal{D}^e_R[\theta_1, \theta_p, \ldots, \theta_{p^{e-1}}]\] is isomorphic as a $\mathcal{D}^e_R$-module to \[ P_e \coloneqq C^e(f^{i_1 + i_2p + \ldots + i_ep^{e-1}} M) \cdot \mathcal{D}^e_R/C^e(f^{1 + i_1 + i_2 p + \ldots + i_e p^{e-1}}M) \cdot \mathcal{D}^e_R.\] 
This isomorphism is given as follows. Given $\varphi \in N_e$ that is contained in the generalized $(i_1,\ldots,i_e)$-eigenspace one has $\varphi(rt^m) = 0$ for all $m < p^e$ if $m \neq i_1 + i_2p + \ldots +i_ep^{e-1}$. The image of $\varphi$ under this isomorphism is the map $r \mapsto \varphi(rf^{i_1 + i_2 p + \ldots +i_e p^{e-1}})$.

\begin{Theo}
\label{BSPGeneral1}
Let $R$ be regular essentially of finite type over an $F$-finite field. Let $(M, \kappa)$ be an $F$-regular Cartier module and $f$ an $M$-regular element. Then for all $e \geq 1$ the generalized $(i_1,\ldots, i_e)$-eigenspace of $N_e$ is isomorphic as a $\mathcal{D}^e_R$-module to \[{F^e}^! Q_i M \coloneqq {F^e}^! (\tau(M, f^{(i_1 + i_2p + \ldots + i_ep^{e-1})/p^e})/\tau(M, f^{(1+i_1 + i_2p + \ldots +i_e p^{e-1})/p^e})) \] where $i = (i_1, \ldots, i_e)$.
\end{Theo}
\begin{proof}
Since $N_e \cong P_e$ it suffices to argue that that we have an isomorphism with $P_e$.
Recall that ${F^e}^!$ is just $\Hom_R(F_\ast^e R, \bullet)$. In particular, for an $R$-submodule $A$ of $M$ we have a natural inclusion $\Hom_R(F_\ast^eR, A) \to \Hom_R(F_\ast^e, M)$. Since $M$ is a Cartier module we have the map $C^e\colon M \to {F^e}^!M$ which is adjoint to $\kappa^e$ (see Definition \ref{AdjointMap}). Now \cite[Lemma 4.6]{blicklestaeblerbernsteinsatocartier} identifies ${F^e}^! \kappa^e A$ with the $\mathcal{D}^e_R$-submodule of ${F^e}^! M$ generated by $C^e(A)$.

If we apply this observation with $A = f^{i_1 + i_2 p + \cdot + i_e p^{e-1}}M$ we thus obtain that 
\[C^e(f^{i_1 + i_2 p + \cdot + i_e p^{e-1}}M) \cdot \mathcal{D}^e_R = {F^e}^! \kappa^e f^{i_1 + i_2 p + \cdot + i_e p^{e-1}}M.\]

By our assumption, $M = \tau(M, f^0)$ so that, via Lemma \ref{CartierTauDivisionByP} and Skoda (\cite[Proposition 4.4]{blicklestaeblerfunctorialtestmodules}) the right-hand side is equal to ${F^e}^! \tau(M, f^{(i_1 + i_2p + \cdots +i_e p^{e-1})/p^e})$.

Since $R$ is regular, $F_\ast^e R$ is projective as an $R$-module. Thus ${F^e}^!$ is exact and the proof complete.
\end{proof}

\begin{Def}[{\cite[Definition 3.7]{blicklestaeblerbernsteinsatocartier}}]
\label{def.bernsteinsato}
Let $(M, \kappa)$ be a Cartier module and $f$ an $M$-regular element. Then we define the $e$th Bernstein-Sato polynomial $b^e_{M,f}[s] \in \mathbb{Q}[s]$ associated to $(M, \kappa, f)$ as follows: Let $\Gamma \subseteq \{0, \ldots, p-1\}^e$ be the set of those $i =(i_1, \ldots, i_e) \in \mathbb{F}_p^e$ for which the generalized $i$-eigenspace of $N_e$ is non-trivial. Then we set \[b^e_{M,f}[s] = \prod_{i \in \Gamma^e} \left(s - \left( \frac{i_e}{p} + \cdots +\frac{i_1}{p^e}\right) \right), \] where we lift $\mathbb{F}_p^e$ to $\mathbb{Z}^e$ by choosing representatives in $\{0, \ldots, p-1\}^e$.
\end{Def}

Recall that any $\lambda \in [0,1)$ can be written uniquely as \[ \sum_{i\geq1} \frac{c_i(\lambda)}{p^i},\] with all $c_i(\lambda) \in \{0, \ldots, p-1\}$ and such that infinitely many of the $c_i(\lambda)$ are non-zero. We will refer to this as the \emph{$p$-adic expansion of $\lambda$}. Moreover, one has \[\frac{\lceil \lambda p^e\rceil -1}{p^e} = \sum_{i=1}^e \frac{c_i(\lambda)}{p^i}.\] We call this the \emph{truncated $p$-adic expansion of $\lambda$}.

\begin{Ko}
\label{BernsteinSatoExtension}
Let $R$ be regular essentially of finite type over an $F$-finite field. Let $(M, \kappa)$ be an $F$-regular Cartier module and $f$ an $M$-regular element. Then the roots of the Bernstein-Sato polynomials $b^e_{M,f}(s)$ are given for all $e \gg 0$ by $\frac{\lceil \lambda p^e \rceil -1}{p^e}$, where $\lambda$ varies over the $F$-jumping number of the test module filtration $\tau(M, f^\mu)$ for $\mu \in (0,1]$.
\end{Ko}
\begin{proof}
By definition, $\lambda$ is an $F$-jumping number if and only if $\tau(M, f^{\lambda}) \neq \tau(M, f^{\lambda - \eps})$ for all $0 < \eps \ll 1$. Using right-continuity of $\tau$, the fact that $\frac{\lceil \lambda p^e\rceil}{p^e} \geq \lambda$ and $\frac{\lceil \lambda p^e\rceil}{p^e} \to \lambda$ for $e \to \infty$ we equivalently have
\[\tau(M, f^{\frac{\lceil \lambda p^e\rceil -1}{p^e}}) \neq \tau(M, f^{\frac{\lceil \lambda p^e\rceil}{p^e}})\] for all $e \gg 0$. Writing \[\frac{\lceil \lambda p^e \rceil -1}{p^e} = \frac{c_e(\lambda) + c_{e-1}(\lambda) + \cdots + c_1(\lambda) p^{e-1}}{p^e}\] using Theorem \ref{BSPGeneral1} and the fact, that ${F^e}^!$ is fully faithful we obtain that the generalized $(c_e(\lambda), \ldots, c_1(\lambda))$-eigenspace of $N_e$ is non-trivial for all $e \gg 0$ if and only if $\lambda$ is an $F$-jumping number.
\end{proof} 

\begin{Bem}
Of course, the point is that $F$-regularity here is meant in the sense of  \cite{blicklestaeblerfunctorialtestmodules} which is a strictly weaker notion than the one used in \cite{blicklestaeblerbernsteinsatocartier}.

Moreover, the use of Lemma \ref{CartierTauDivisionByP} and Skoda in Theorem \ref{BSPGeneral1} removes an inaccuracy in the proof of \cite[Theorem 5.4]{blicklestaeblerbernsteinsatocartier}. Indeed, this guarantees an equality $\kappa^e f^m M = \tau(M, f^{\frac{m}{p^e}})$ for any $m \in \mathbb{Z}$ and $e \geq 1$. The application of \cite[Corollary 4.8]{blicklestaeblerbernsteinsatocartier} in the proof of \cite[Theorem 5.4]{blicklestaeblerbernsteinsatocartier} is not correct since one may have to enlarge the numerator.

We will prove a partial generalization of this result where we relax the $F$-regularity assumption in Section \ref{BspSigmaSection}.
\end{Bem}

Corollary \ref{BernsteinSatoExtension} implies that the limit $b_{M,f}(s) \in \mathbb{Q}[s]$ of the sequence $(b^e_{M,f}(s))_e$ exists and that $b_{M,f}(s)$ is the product over the $s - \lambda$ where $\lambda \in [0,1)$ runs over the $F$-jumping numbers of $(M, \kappa)$ filtered along $f$. The sequence of modules ${F^e}^! \gamma_\ast M$ form a direct system (see Lemma \ref{directedsystem}). In this way one obtains a $\mathcal{D}_{R[t]}$-module $\colim_{e \geq 0} {F^{e}}^! \gamma_\ast M$. Denote the natural map $M \to \colim_{e \geq 0} {F^{e}}^! \gamma_\ast M$ by $\varphi_0$. Then one could replace $N_e$ by considering the quotient
\[ \varphi_0(M)\mathcal{D}_{R[t]}/\varphi_0(fM)\mathcal{D}_{R[t]}.\]
It was observed in \cite[Example 6.15]{mustatabernsteinsatopolynomialspositivechar} that one loses information in this way. In \cite{bitounbernsteinsatoposchar} it was shown that for $M = R$ the Bernstein-Sato polynomial one obtains by working in the limit only records those $F$-jumping numbers whose denominator is not divisible by $p$. We will extend this to $F$-regular Cartier modules $(M,\kappa)$.

Once again we remind the reader about the $N_e$ which were defined just before Theorem \ref{BSPGeneral1}.
\begin{Prop}
\label{EigenspaceRestrictionCartier}
Let $M$ be a Cartier module on $R$ and $f \in R$ a non-zero-divisor on $M$. Then the transition map $\gamma_\ast {F^{e}}^! M \to {F^{e+a}}^! \gamma_\ast M$ induces a map $N_{e} \to N_{e+a}$. Furthermore, if $E_i$ is the generalized $(i_1, \ldots, i_e)$-eigenspace of $N_e$ and $E_{i,j}$ the generalized $(i_1, \ldots, i_e, j_{e+1}, \ldots, j_{e+a})$- eigenspace of $N_{e+a}$ then we get an induced map
\[\alpha\colon \begin{xy}  \xymatrix@1{E_i \ar[r]^\iota & N_e \ar[r]& N_{e+a} \ar[r]^{\pi} & E_{i,j},}\end{xy}\] where $\pi$ is the projection onto the generalized eigenspace and $\iota$ is the natural inclusion.

If, in addition, $M$ is $F$-regular, then via the isomorphisms $E_i \to {F^e}^! Q_i M$, $E_{i,j} \to {F^{e+a}}^! Q_{i,j} M$ (cf.\ the discussion preceding and including Theorem \ref{BSPGeneral1}) and adjunction we get an induced Cartier structure $F_\ast^{a} {F^e}^! Q_i M \to {F^e}^! Q_i M$ which is given by \[\varphi \longmapsto \kappa^a f^{j_{e+1} + j_{e+2} p + \ldots + j_{e+a}p^{e+a-1}} \circ \varphi.\]
\end{Prop}
\begin{proof}
Using Lemma \ref{directedsystem} we have a commutative diagram
\[\begin{xy} \xymatrix{\gamma_\ast {F^e}^! M \ar[r] \ar[d]& {F^e}^! \gamma_\ast M\ar[d]^\beta \\ \gamma_\ast {F^{e+a}}^! M \ar[r] & {F^{e+a}}^! \gamma_\ast M}\end{xy} \]
Thus we have $\beta(\gamma_\ast C^e(M)) = \gamma_\ast C^{e+a}(M)$. Since the map $\beta$ is $\mathcal{D}^e_R[\theta_1, \theta_p, \ldots, \theta_{p^{e-1}}]$-linear we obtain an inclusion
\[\beta(\gamma_\ast C^e(M)\mathcal{D}^e_R[\theta_1, \theta_p, \ldots, \theta_{p^{e-1}}]) \subseteq \gamma_\ast C^{e+a}(M) \mathcal{D}^{e+a}_R[\theta_1, \theta_p, \ldots, \theta_{p^{e+a-1}}] \] and an inclusion
\[\beta(\gamma_\ast C^e(fM)\mathcal{D}^e_R[\theta_1, \theta_p, \ldots, \theta_{p^{e-1}}]) \subseteq \gamma_\ast C^{e+a}(fM) \mathcal{D}^{e+a}_R[\theta_1, \theta_p, \ldots, \theta_{p^{e+a-1}}]. \] In this way we obtain the induced map $\alpha$.

Recall from the discussion preceding and including Theorem \ref{BSPGeneral1} that the isomorphism $E_i \to {F^e}^! Q_i M$ is given by \[\varphi \longmapsto [r \mapsto \varphi(r f^{i_1 +i_2p + \ldots i_ep^{e-1}})].\] Similarly, the image of $\alpha(\varphi)$ in ${F^{e+a}}^! Q_{i,j} M$ is given by \[r \longmapsto \pi \kappa^a \varphi(rf^{i_1 + i_2 p + \ldots i_e p^{e-1}}).\]

\begin{claim}
\label{claim.maps}
The image of $\alpha(\varphi)$ in ${F^{e+a}}^! Q_{i,j} M$ coincides with the map 
\[\psi_j\colon r \longmapsto \kappa^a \varphi(r f^{i_1 + i_2 p + \ldots i_e p^{e-1} + j_{e+1} p^e + \ldots j_{e+a}p^{e+a-1}}).\]
\end{claim}
\begin{proof}[proof of \ref{claim.maps}.]
Note that  $\{ rt^l \, \vert \, r \in R$ with $0 \leq l \leq p^{e+a} -1$\} is a set of generators for $F_\ast^{e+a} R[t]$.
Since $N_{e+a}$ is the direct sum of its generalized eigenspaces the claim comes down to verifying that the equality \[\sum_{j \in \mathbb{F}_p^a} \psi_j = \kappa^a \varphi\] holds in $N_{e+a}$, where we lift the $\psi_j$ to $N_{e+a}$ by taking the natural section of the projection onto the generalized eigenspace.

Hence, $\psi_j$ evaluated at \[t^{i_1 + i_2 p + \ldots + i_ep^{e-1} + l_{e+1} p^e + \ldots l_{e+a} p^{e+a -1}}\] coincides with \[\kappa^a\varphi(f^{i_1 + i_2 p + \ldots + i_ep^{e-1} + l_{e+1} p^e + \ldots l_{e+a} p^{e+a -1}})\] if $j = l$ and is zero otherwise. In particular, we see that the claimed identity holds.
\end{proof}

Identifying $\varphi$ with its image in ${F^e}^! Q_i M$ we may view $\varphi$ as a map $\varphi\colon F_\ast^e R \to M$. Then \[\psi_j(r) = \kappa^a f^{j_{e+1} + j_{e+2} p + \ldots j_{e+a} p^{a-1}} \varphi(f^{i_1 + i_2 p + \ldots + i_e p^{e-1}}r )\] which shows that the Cartier structure has the desired form.
\end{proof}

\begin{Bem}
We note that taking the projection onto the corresponding generalized eigenspace in Proposition \ref{EigenspaceRestrictionCartier} is necessary (see \cite[Example 6.15]{mustatabernsteinsatopolynomialspositivechar}).
\end{Bem}

If $\lambda \in (0,1]$ with truncated $p$-adic expansion \[\sum_{i=1}^e \frac{c_i(\lambda)}{p^i},\] then we have seen in the proof of Corollary \ref{BernsteinSatoExtension} that for all $e \gg 0$ the ${F^e}^! Q_{(c_e(\lambda), \ldots, c_1(\lambda))} M$ coincide with ${F^e}^! Gr^\lambda M$, where $Gr^\lambda M = \tau(M, f^{\lambda})/\tau(M, f^{\lambda - \eps})$.

\begin{Le}
\label{SufficientCartier}Let $(M, \kappa)$ be an $F$-regular Cartier module. For $e \gg 0$ the morphism $\psi$ given by the composition\[ \begin{xy}\xymatrix@1{{F^e}^! Gr^{\lambda} M \ar@/^1.5pc/[rrr]^{\psi}\ar[r]^-\cong& E_i \ar[r]^\alpha& E_{i,j} \ar[r]^-\cong & {F^{e+a}}^! Gr^{\lambda} M} \end{xy},\] where $\alpha$ is the map in Proposition \ref{EigenspaceRestrictionCartier}, is of the form ${F^e}^! C$ for a morphism $C\colon Gr^\lambda M \to {F^a}^! Gr^\lambda M$ if $\lambda (p^s -1)$ is an integer for some $s \in \mathbb{Z}$ and both $e,a$ are multiples of $s$.
\end{Le}
\begin{proof} There is an integer $s$ such that $\lambda(p^s -1)$ is an integer if and only if the denominator of $\lambda$ is not divisible by $p$. This is equivalent to the fact that the $p$-adic expansion of $\lambda$ is strictly periodic with period length dividing $s$. In this case, one has $\lambda(p^{s}-1) = p^{s} \sum_{i=1}^{s} \frac{c_i(\lambda)}{p^i}$ (see e.g.\ \cite[Lemma 2.6]{hbwzfthresholdshomogeneouspoly}). Write $e = e' s$ and $a = a' s$ Then by Proposition \ref{EigenspaceRestrictionCartier} and the previous observation the Cartier structure ${F^e}^! Gr^\lambda M \to {F^{e+a}}^! Gr^\lambda M$ is given by the adjoint of \[\kappa^a f^{\lceil \lambda (p^a -1) \rceil} = (\kappa^s f^{\lceil \lambda (p^s -1) \rceil})^{a'}\] and we may set \[C = [m \longmapsto [r  \mapsto (\kappa^s f^{\lceil \lambda (p^s-1) \rceil})^{e'}(rm)]].\]
\end{proof}

\begin{Bem}
Note that the Cartier structure on ${F^e}^! Gr^\lambda M$ is only giving us information on $Gr^\lambda M$ if we have a map $C$ as in Lemma \ref{SufficientCartier}. Quite generally, in this case, $C$ is a nil-isomorphism (see e.g.\ \cite[Lemma 2.2]{staeblerunitftestmoduln}) so that as crystals $Gr^\lambda M$ and ${F^e}^! Gr^\lambda M$ coincide. Otherwise ${F^e}^! Gr^\lambda M$ does, at least a priori, not encode more information than any other faithful functor.
\end{Bem}

We come to the main result of this section:

\begin{Theo}
\label{AnotherMainResult}
Let $R$ be regular and essentially of finite type over an $F$-finite field, $(M, \kappa)$ an $F$-regular Cartier module and $f$ an $M$-regular element. Assume that $\lambda \in \mathbb{Q}$ is an $F$-jumping number of the test module filtration of $M$ along $f$. Then for $e \gg 0$ we consider the map $C\colon {F^e}^! Gr^\lambda M \to {F^{e+a}}^! Gr^\lambda M$ (i.e.\ the morphism $\psi$ of Lemma \ref{SufficientCartier}).

If $\lambda (p^s -1) \in \mathbb{Z}$, then $C$ admits a factorization \[\begin{xy} \xymatrix{ Gr^\lambda M \ar[d]^{D^u} \ar[dr]^{D^v}& \\ {F^e}^!Gr^\lambda M \ar[r]^C & {F^{e+a}}^! Gr^\lambda M}\end{xy} \] where $D\colon Gr^\lambda M \to {F^s}^! Gr^\lambda M$ is adjoint to a map $\kappa^s f^c\colon F^s_\ast M \to M$, if and only if $s \mid e$ and $s \mid a$. In this case, $c = \lambda(p^s -1)$ and $\kappa^s f^c$ is not nilpotent.

If $\lambda(p^s -1) \notin \mathbb{Z}$, then there is a morphism $D\colon Gr^\lambda M \to {F^s}^! Gr^\lambda M$, that is adjoint to a map $\kappa^s f^c$, fitting into a commutative diagram as above if and only if $s \mid e$, $s = a$ and $\lceil \lambda (p^s - 1) \rceil$ is of the form $c_1 p^{s-1} + c_2 p^{s-2} + \ldots + c_s$, where $\frac{c_1}{p} + \frac{c_2}{p^2} + \ldots + \frac{c_s}{p^s}$ is the truncated $p$-adic expansion of $\lambda$. In this case, the morphism $D$ is nilpotent and given by the adjoint of $\kappa^s f^{\lceil \lambda(p^s -1) \rceil }$.
\end{Theo}
\begin{proof}
We start with the case that $\lambda(p^s -1)$ is an integer. Clearly, if $C$ admits such a factorization then $us = e$ and $vs = e +a$. The converse is Lemma \ref{SufficientCartier}. To see that it is not nilpotent we may use Lemma \ref{Remarkable}.

We now turn to the case that $\lambda(p^s -1)$ is not an integer. Then it is clearly necessary that $s \mid e$ and $s \mid a$. If we write the $p$-adic expansion of $\lambda$ as $\lambda = \sum_{i \geq 1} c_i p^{-i}$, then the existence of such a morphism $\kappa^s f^{u_s(\lambda)}\colon Gr^\lambda M \to {F^s}^! Gr^\lambda M$ is equivalent to a commutative diagram of the form
\[\begin{xy} \xymatrixcolsep{7pc}\xymatrix{Gr^\lambda M \ar@/^3pc/[rr]^{\kappa^{e+a} f^{u_s(\lambda) \frac{p^{e+a} -1}{p^s-1}}} \ar[r]^{\kappa^e f^{u_s(\lambda) \frac{p^{e}-1}{p^s-1}}}& {F^e}^! Gr^\lambda M \ar[r]^{\kappa^a f^{c_a + c_{a-1}p + \ldots + c_1 p^{a-1}}}& {F^{e+a}}^! Gr^\lambda M.}\end{xy}\]
By Lemma \ref{Remarkable} we may write $u_s(\lambda) = \lceil \lambda( p^s -1)\rceil + m$ for some non-negative integer $m$. Using this, the diagram is commutative if and only if we have the following equality
\[(\lceil \lambda (p^s -1) \rceil + m) \frac{p^e -1}{p^s -1} p^a + c_a + c_{a-1} p + \ldots + c_1 p^{a-1} = (\lceil \lambda (p^s -1) \rceil +m) \frac{p^{e+a} -1}{p^s -1}.\]
This simplifies to
\[c_a + c_{a-1} p + \ldots + c_1 p^{a-1} = (\lceil \lambda(p^s -1) \rceil +m)\frac{p^a -1}{p^s -1}.\] Note that $\lceil \lambda(p^s -1) \rceil = c_1 p^{s-1} + c_2 p^{s-2} + \ldots + c_s + \eps$, where $\eps \in \{0,1\}$ depending on whether the fractional part of $\lambda(p^s -1)$ is negative or positive. Using this formula and writing $a =s \tilde{a}$ so that $\frac{p^a-1}{p^s-1} = \sum_{i=0}^{\tilde{a} -1} p^{is}$ we may further expand the equation to
\begin{equation*}\begin{split} c_a + c_{a-1} p + \ldots + c_1 p^{a-1} ={} &c_1 p^{a-1} + c_2 p^{a-2} + \ldots + c_s p^{a-s} + (m + \eps )p^{a-s} \\
& +c_1 p^{a-1-s}  + c_2 p^{a-2-s} + \ldots + c_s p^{a-2s} + (m + \eps) p^{a-2s}\\
& \,\,\,\vdots\\
& +c_1 p^{s-1} + c_2  p^{s-2} + \ldots + c_s + (m + \eps)\end{split}\end{equation*} where we have $\tilde{a}$ rows. The summands in the first row of the right hand side except the last one all occur on the left hand side. Subtracting this from both sides we observe that the left hand side is $< p^{a-s}$. Hence, if $m + \eps$ is positive, then equality cannot hold. Thus, we must have $m = \eps = 0$.

If $m = \eps = 0$, then the last summand of each row vanishes. If moreover $a =s$, then equality holds. If $s < a$, then there must exist $0 \leq i \leq s$ with $c_i \neq c_{rs + i}$ for some $1 \leq r < \tilde{a}$ . Hence, equality cannot hold.
\end{proof}

\begin{Bsp}
The second case of the theorem occurs for instance for any $F$-jumping number of the form $\frac{ap}{p^2 -1}$ with $a < p$ (and necessarily $p > 2$).

If the denominator of $\lambda$ is divisible by $p$ then the second case of the theorem can only occur if $\lambda = \frac{a}{p^nb}$ with $p \nmid b$ and $a > b$. Indeed, assume that $a \leq b$. Denote the coefficients of the $p$-adic expansion of $\lambda$ by $c_i$ and the coefficients of the expansion of $\frac{a}{b}$ by $d_i$.
Then $c_1 = \ldots = c_n =0$ and $c_{n+i} = d_i$ for $i \geq 1$. In particular, the fractional part of $\lambda(p^s-1)$ is then non-negative since it is of the form \[\sum_{i \geq 1} \frac{c_i}{p^{i-s}} - \sum_{i \geq 1}\frac{c_i}{p^i}.\] Since $\lambda(p^s -1)$ is not an integer by assumption its fractional part is strictly positive so that the second case of the theorem cannot occur.
\end{Bsp}

\begin{Que}
Are there $F$-pure thresholds which actually satisfy the second case of Theorem \ref{AnotherMainResult}?
\end{Que}

A possible strategy may be to consider quasi-homogeneous polynomials $f \in k[x_1, \ldots, x_n]$ whose Jacobian ideal coincides, up to radical, with the irrelevant ideal. For then a likely candidate for the $F$-pure threshold is the log canonical threshold \[\lct(f) = \min\{\frac{\sum \deg(x_i)}{\deg f}, 1\}.\] If $\fpt(f) \neq \lct(f)$ then the denominator of the $F$-threshold has to be a $p$th power (\cite[Theorem 3.5]{hbwzfthresholdshomogeneouspoly}). So one would need $f$ with $\fpt(f) = \lct(f)$.

Given a Cartier module $(M, \kappa)$ we have, after passing to the graph embedding, natural maps $\gamma_\ast M \to {F^e}^! \gamma_\ast M$. Taking the colimit over these maps we obtain $\colim_{e \geq 0} {F^e}^! \gamma_\ast M$ which is now a $\mathcal{D}_{R[t]}$-module. In fact, one has $\colim_{e \geq 0} {F^e}^! \gamma_\ast M = \gamma_+ \colim_{e \geq0} {F^e}^!M$, where $\gamma_+$ is the (right) $\mathcal{D}$-module pushforward (see \cite[Proposition 3.5]{blicklestaeblerbernsteinsatocartier}). Write $\varphi_0$ for the natural map $\gamma_\ast M \to \colim_{e \geq 0} {F^e}^! \gamma_\ast M$. Then, instead of studying the generalized eigenspaces of the $N_e$ (defined just before Theorem \ref{BSPGeneral1}), we may also study the generalized eigenspaces of the following quotient
\[S_e \coloneqq \varphi_0(\gamma_\ast M) \mathcal{D}^e_R[\theta_1, \theta_p, \ldots, \theta_{p^{e-1}}]/\varphi_0(\gamma_\ast fM) \mathcal{D}^e_R[\theta_1, \theta_p, \ldots, \theta_{p^{e-1}}]. \]
Loosely speaking, the ambient $\mathcal{D}^e_{R[t]}$-module ${F^e}^! \gamma_\ast M$ is replaced by the full $\mathcal{D}_{R[t]}$-module $\colim_{e \geq0} {F^e}^! \gamma_\ast M$ but we still only consider the $\mathcal{D}^e_R[\theta_1, \theta_p, \ldots, \theta_{p^{e-1}}]$-module generated by $\varphi_0(\gamma_\ast M)$ or $\varphi_0(\gamma_\ast fM)$.

\begin{Bem}
Working with $\colim_{e \geq 0} {F^e}^! \gamma_\ast M$ for $M = \omega_R$ is the setting of \cite{mustatabernsteinsatopolynomialspositivechar} (and then also to passing to left $\mathcal{D}$-modules). Indeed, note that by \cite[Proposition 3.5]{blicklestaeblerbernsteinsatocartier} one has $\colim_{e \geq 0} {F^e}^! M = \gamma_+ \colim_{e \geq 0} {F^e}^! M$ as right $\mathcal{D}_{R[t]}$-modules, where $\gamma_+$ is the $\mathcal{D}$-module pushforward. If $M = \omega_R$, then $\colim_{e \geq 0} {F^e}^! \omega_R \cong \omega_R$. Here, by convention, $\omega_R \to {F^e}^! \omega_R$ is induced via $f^!$ from a fixed isomorphism $k \to F^! k$, where $f\colon \Spec R \to k$ is the structural map\footnote{Of course, if $k$ is perfect, then $k \to F^! k$ is canonical and $\omega_R \to F^! \omega_R$ is the Cartier isomorphism.}. The equivalence between right and left $\mathcal{D}_{R[t]}$-modules then sends $\gamma_+ \omega_R$ to $\gamma_+ R = R[t]_{t-f}/R[t]$.
\end{Bem}

Our next goal is to point out that the generalized eigenspaces of the $N_e$ and of the $S_e$ contain the same information. This will be accomplished in two steps. First we prove this in the case that the maps $\gamma_\ast M \to {F^e}^! \gamma_\ast M$ are all injective (Theorem \ref{Herrgottssack} below). This can always be arranged by replacing $M$ by the nil-isomorphic object $\overline{M}$ (see Definition \ref{Mbar} below).

Then we show that given a nil-isomorphism $\varphi\colon M \to N$ of Cartier modules the restriction of $\varphi$ induces a map $Gr^\lambda M \to Gr^\lambda N$ which is then again a nil-isomorphism with respect to the Cartier structures defined in Lemma \ref{Remarkable}

Before we proceed, we need to recall one more notion from the theory of Cartier crystals.

\begin{Def}
\label{Mbar}
Given a a Cartier module $(M, \kappa)$ we can consider the union of all nilpotent submodules $M_{\text{nil}}$. This is again a nilpotent Cartier module and we define $\overline{M} = M/M_{\text{nil}}$. Note, in particular, that the natural projection $M \to \overline{M}$ is a nil-isomorphism.
\end{Def}

We note the following

\begin{Le}
\label{NilpotentsPushforward}
Let $i\colon \Spec R/I \to \Spec R$ be a closed immersion and $(M, \kappa)$ a Cartier module. Then $i_\ast \overline{M} = \overline{i_\ast M}$.
\end{Le}

\begin{Theo}
\label{Herrgottssack}
Assume that $(M, \kappa)$ is an $F$-regular Cartier module satisfying $\overline{M} = M$. Write $\gamma_+ \mathcal{M} = \colim_{e \geq 0} {F^e}^! \gamma_\ast M$ and denote the natural maps ${F^e}^! \gamma_\ast M \to \gamma_+ \mathcal{M}$ by $\varphi_e$. Then for any $e$ the quotients  \[S_e = \varphi_0(\gamma_\ast M) \mathcal{D}^e_R[\theta_1, \theta_p, \ldots, \theta_{p^{e-1}}]/\varphi_0(\gamma_\ast fM) \mathcal{D}^e_R[\theta_1, \theta_p, \ldots, \theta_{p^{e-1}}] \] and \[N_e = \gamma_\ast C^e(M) \mathcal{D}^e_R[\theta_1, \theta_p, \ldots, \theta_{p^{e-1}}]/\gamma_\ast C^e(fM) \mathcal{D}^e_R[\theta_1, \theta_p, \ldots, \theta_{p^{e-1}}]\] are naturally isomorphic as $\mathcal{D}^e_R[\theta_1, \theta_p, \ldots, \theta_{p^e-1}]$-modules. In particular, we obtain $\mathcal{D}^e_R$-isomorphisms of generalized eigenspaces.
\end{Theo}
\begin{proof}
First of all, note that the natural map $\varphi_e\colon {F^e}^! \gamma_\ast M \to \gamma_+ \mathcal{M}$ is $\mathcal{D}^e_{R[t]}$-linear.

Next, observe that by Lemma \ref{NilpotentsPushforward} we have $\overline{\gamma_\ast M} = \gamma_\ast M$. It follows that the map $\gamma_\ast M\to {F^e}^! \gamma_\ast M$ is injective for all $e$ (cf. \cite[Lemma 6.20]{staeblertestmodulnvilftrierung}). Hence, the natural map $\varphi_0$ is also injective. Note that $\varphi_0$ factors as $\gamma_\ast M \xrightarrow{\gamma_\ast C^e} {F^e}^! \gamma_\ast M \xrightarrow{\varphi_e} \gamma_+ \mathcal{M}$ for any $e \geq 1$. Since $\varphi_e$ is $\mathcal{D}^e_{R[t]}$-linear we have \[\varphi_e(\gamma_\ast C^e(M) \mathcal{D}^e_R[\theta_1, \theta_p, \ldots, \theta_{p^{e-1}}]) = \varphi_0(\gamma_\ast C^e(M)) \mathcal{D}^e_R[\theta_1, \theta_p, \ldots, \theta_{p^{e-1}}],\] and a similar assertion holds for $\gamma_\ast C^e(fM)$.
We conclude that the quotients \[\varphi_0(\gamma_\ast C^e(M)) \mathcal{D}^e_R[\theta_1, \theta_p, \ldots, \theta_{p^{e-1}}]/\varphi_0(\gamma_\ast C^e(fM)) \mathcal{D}^e_R[\theta_1, \theta_p, \ldots, \theta_{p^{e-1}}] \] and \[\gamma_\ast C^e(M) \mathcal{D}^e_R[\theta_1, \theta_p, \ldots, \theta_{p^{e-1}}]/\gamma_\ast C^e(fM) \mathcal{D}^e_R[\theta_1, \theta_p, \ldots, \theta_{p^{e-1}}]\]are naturally isomorphic as $\mathcal{D}^e_R[\theta_1, \theta_p, \ldots, \theta_{p^e-1}]$-modules since $\varphi_0$ is injective. In particular, we obtain $\mathcal{D}^e_R$ isomorphisms of generalized eigenspaces.
\end{proof}

\begin{Prop}
Let $(M, \kappa_M)$, $(N, \kappa_N)$ be Cartier modules and $f \in R$ $M$-regular and $N$-regular. If $\varphi\colon M \to N$ is a nil-isomorphism then we get an induced nil-isomorphism $Gr^\lambda M \to Gr^\lambda N$ with respect to any of the Cartier structures $\kappa_?^e f^{a}$, where $? = M$ or $N$ and $a \geq \lceil \lambda(p^e -1)\rceil$.
\end{Prop}
\begin{proof}
By \cite[Theorem 2.8]{staeblerunitftestmoduln} the restriction of $\varphi$ induces a surjective map $\tau(\varphi)\colon \tau(M, f^\lambda) \to \tau(N, f^\lambda)$. By \cite[Lemma 2.1, proof of Theorem 2.8]{staeblerunitftestmoduln} we may assume that both $M, N$ are $F$-pure. In particular, $\varphi$ is surjective. Now the claim follows from \cite[Theorem 5.8]{staeblertestmodulnvilftrierung}.
\end{proof}

In particular, given an $F$-regular Cartier module $M$ we obtain nil-isomorphisms $Gr^\lambda M \to Gr^\lambda \overline{M}$. Since for $\overline{M}$ working on the $e$th level or in the colimit induces natural isomorphisms of the generalized eigenspaces of the quotients by Theorem \ref{Herrgottssack} we see that we obtain the same nilpotence results if we construct the quotients by working with the colimit.

\section{Non-$F$-pure modules}
\label{Nonfpuremodules}
In this section we study a generalization of the \emph{non-$F$-pure ideal} or \emph{$\phi$-fixed ideal} to modules. These non-$F$-pure ideals were first introduced in \cite{fujinotakagischwedenonlc} and are further developed and studied in \cite{schwedecanlinearsystem} and \cite{hsiaoschwedezhangtoriccartier}. They are the characteristic $p$ analog of the so-called \emph{non-lc ideal} and have applications to birational geometry.

The importance of these for us is that we will prove a connection between zeros of Bernstein-Sato polynomials and certain non-$F$-pure-modules. As an other application we will partially answer a question of Bhatt, Schwede and Takagi \cite[Question 5.6]{bhattschwedetakagiweakordinarityfsing} in \cite{staeblerordinarynonlc}.

These non-$F$-pure modules will form a non-increasing and discrete filtration of the ambient Cartier module. However, there will be no continuity properties even if $M$ is $F$-regular.

The actual connection with Bernstein-Sato polynomials will be discussed in section \ref{BspSigmaSection}. Here we just develop the basic theory of non-$F$-pure modules.

\begin{Assu}
\label{assumption}
Throughout this section we assume that $R$ is an $F$-finite ring, $(M, \kappa)$ is a Cartier module and $f \in R$ is an $M$-regular element.
\end{Assu}

On several occasions we will need existence of test modules $\tau(M, f^\lambda)$. In these cases we will ensure this by making the additional assumption that $R$ is essentially of finite type over an $F$-finite field.

\begin{Def}
Given $(M, \kappa)$, an $M$-regular element $f$ and a non-negative rational number $\lambda$ we denote by $\mathcal{C}$ the Cartier algebra \[ \bigoplus_{e \geq 0} \kappa^e f^{\lceil \lambda (p^e -1)\rceil}R. \] We define the \emph{non-$F$-pure submodule} $\sigma(M, f^\lambda)$ as $\mathcal{C}_+^h M$ for all $h \gg 0$.
\end{Def}

Note that by \cite[Proposition 2.13]{blicklep-etestideale} the descending chain $\mathcal{C}_+^h M \supseteq \mathcal{C}_+^{h+1} M \supseteq \ldots$ stabilizes for \emph{arbitrary} Cartier algebras in any $F$-finite ring, i.e.\ $\sigma(M,f^\lambda)$ is well-defined.

We will also use the notation $\underline{M}_\mathcal{C}$ for $\mathcal{C}_+^h M$ ($h \gg 0$) or even $\underline{M}$ if $\mathcal{C}$ is clear from context.

\begin{Bem}
As is well-known to experts one can compute $\tau(M, f^\lambda)$ using either maps of the form $\kappa^e f^{\lceil \lambda p^e \rceil}$ or $\kappa^e f^{\lceil \lambda (p^e-1)\rceil}$.

This his however not the case for non-$F$-pure ideals or submodules (also well-known to experts). For completeness sake we illustrate this with a simple example: Take $R = M = \mathbb{F}_p[x]$ and $\lambda =1$ and for $\kappa$ the usual generator of $\Hom_R(F_\ast R, R)$. Then $\sigma(M, f^1) = R$, whereas if we consider the Cartier algebra $\mathcal{E}$ generated in degree $e$ by $\kappa^e f^{\lceil \lambda p^e \rceil}$, then we obtain $\mathcal{E}_+^h R = \tau(R, x^1) = (x)$.
\end{Bem}

\begin{Prop}
\label{SigmaCrystal}
We have $\sigma(M, f^\lambda) = \sigma(\sigma(M, f^\mu), f^\lambda)$ for all $0 \leq \mu \leq \lambda$. Moreover, if $\varphi\colon (M, \kappa) \to (N, \kappa')$ is a nil-isomorphism, then $\varphi$ induces a nil-isomorphism $\sigma(M, f^\lambda) \to \sigma(N,f^\lambda)$ and $\varphi(\sigma(M, f^\lambda)) = \sigma(N, f^\lambda)$.
\end{Prop}
\begin{proof}
The first assertion follows from \cite[Lemma 2.1]{staeblerunitftestmoduln}.
For the second assertion denote by $\mathcal{C}$ the graded subring of the non-commutative polynomial ring quotient \[R\{F\}/\langle rF^e - F^e r^{p^e} \vert r \in R, e \in \mathbb{N}\rangle\] given in degree $e$ by $F^e f^{\lceil \lambda (p^e -1) \rceil} R$. Then we let $\mathcal{C}$ act on $M$ by sending $F$ to $\kappa$ and similarly obtain an action on $N$ by sending $F$ to $\kappa'$. 

To say that $\varphi$ is a nil-isomorphism then means that $\varphi(F^e \cdot m) = F^e \cdot \varphi(m)$ and that both the kernel and the cokernel of $\varphi$ are annihilated by some power of $F$. Thus a fortiori $\ker \varphi$ and $\coker \varphi$ are annihilated by $\mathcal{C}_+^e \subseteq \mathcal{C}_{\geq e}$ for all $e \gg 0$. Thus $\varphi$ is also a $\mathcal{C}$-nil-isomorphism.

We compute \[ \varphi(\underline{M}_{\mathcal{C}}) = \varphi(\mathcal{C}_+^e M) = \mathcal{C}_+^e \varphi(M) \subseteq \mathcal{C}_+^e N = \underline{N}_{\mathcal{C}}\] for all $e \gg 0$. Since $\varphi$ and the natural inclusions $\underline{N}_\mathcal{C} \to N$, $\underline{M}_\mathcal{C} \to M$ are all $\mathcal{C}$-nil-isomorphisms we deduce that map induced by restriction \[\psi\colon \underline{M}_\mathcal{C} \longrightarrow \underline{N}_\mathcal{C}\] is also a $\mathcal{C}$-nil-isomorphism. By definition $\sigma(N, f^\lambda)$ does not admit any nilpotent quotients. Hence, we must have equality as claimed.
\end{proof}

\begin{Ko}
\label{SigmaWlogFPure}
Let $(M, \kappa)$ be a Cartier module. Then $\sigma(M, f^\lambda) = \sigma(\underline{M}_\kappa, f^\lambda)$, i.e.\ we may always assume that $(M, \kappa)$ is $F$-pure when computing $\sigma(M, f^\lambda)$.
\end{Ko}
\begin{proof}
The inclusion $\underline{M}_\kappa \to M$ is a nil-isomorphism. Now use Proposition \ref{SigmaCrystal}.
\end{proof}

\begin{Prop}
\label{DiscreteDecreasing}
The filtration $\sigma(M, f^\lambda)$ is non-increasing. If $R$ is of finite type over an $F$-finite field, then it is also discrete.
\end{Prop}
\begin{proof}
In order to show that the filtration is non-increasing simply observe that for any $e$ one has $\prod_{i=1}^e \kappa^{e_i} f^{\lceil (\lambda +\eps)(p^e -1) \rceil} M \subseteq \prod_{i=1}^e \kappa^{e_i} f^{\lceil \lambda(p^e-1)\rceil}M$. If $R$ is of finite type over an $F$-finite field, then discreteness follows by the same argument as in \cite[Theorem 4.18]{blicklep-etestideale}.
\end{proof}

As is well-known to experts one cannot expect any continuity properties for the filtration $\sigma(M, f^\lambda)$ (not even in the case $M = R$). For completeness we give an example (this also readily follows from Proposition \ref{SigmaEqualsTauFregularLocus} and Corollary \ref{Sigmataupdenominator} below):

\begin{Bsp}
Consider the $\mathbb{F}_p[x]$-submodule $M = \mathbb{F}_p[x] \cdot x^{-1} \subseteq \mathbb{F}_p[x, x^{-1}]$. Consider the usual Cartier action on $\mathbb{F}_p[x]$ i.e.\ $\kappa(x^i) = 1$ if $i = p-1$ and $= 0 $ for $0 \leq i \leq p-2$. This action then extends to $M$ via the formula \[\kappa(\frac{g}{x}) = \frac{1}{x} \kappa(gx^{p-1}).\]

Take $f= x^2, \lambda = \frac{1}{2}$ and assume that $p \neq 2$. Then we claim that $\sigma(M, f^\lambda) = k[x]$. We have \[\mathcal{C}_+ M = \sum_{e \geq 1} \kappa^e x^{2 (\frac{1}{2}(p^e -1)) } x^{-1} \mathbb{F}_p[x] = \sum_{e \geq 1} \kappa^e x^{p^e-1} x^{-1} \mathbb{F}_p[x] = \mathbb{F}_p[x].\]
One readily sees that $\mathcal{C}_+ \mathbb{F}_p[x] = \mathbb{F}_p[x]$. Hence, $\sigma(M, f^\lambda) = \mathcal{C}_+^e M = \mathcal{C}_+^{e-1} \mathbb{F}_p[x] = \mathbb{F}_p[x]$ for all $e \gg 0$.

Next, we claim that $\sigma(M, f^{\lambda + \eps}) \subseteq (x)$ for any $0 < \eps \ll 1$. Since $\mathcal{C}_+^h \subseteq \mathcal{C}_{\geq h}$ it suffices to show that $\mathcal{C}_{\geq h} M \subseteq (x)$ for all $h \gg 0$. Thus we have to check that for all $h \gg 0$ \[\kappa^h x^{2 \lceil (\frac{1}{2} + \eps)(p^h -1)\rceil} x^{-1} \mathbb{F}_p[x] \subseteq (x)\] 

For this it suffices to show that the exponent of $x$ in the expression above is $\geq p^h$. Write $\lceil (\frac{1}{2} + \eps)(p^h -1) \rceil = (\frac{1}{2} + \eps)(p^h -1) + \delta$ with $\delta \in [0,1)$. Then the exponent takes the following form
\[(1 + 2\eps)(p^h -1) + 2 \delta -1 = p^h + 2 \eps p^h -1 - 2  \eps + 2\delta -1.\] Clearly for $h \gg 0$ (depending on $\eps$) this is larger than $p^h$.
Thus we see that the filtration is, for $p \neq 2$, right-continuous at $\frac{1}{2}$.

Let us now argue that it is left-continuous at $\frac{1}{2}$ for $p  = 2$. More precisely, for $p =2$ we assert that $\sigma(M, f^{\frac{1}{2}}) \subseteq (x)$ and $\sigma(M, f^{\frac{1}{2} - \eps}) = \mathbb{F}_2[x]$ for all $0 < \eps \ll 1$. To see the first inclusion recall that  $\sigma(M, f^\lambda) \subseteq \mathcal{C}_+^2 M$ and note that elements of $\mathcal{C}_+^2 M$ are sums of elements of the form
\[\kappa^a x^{2\lceil 2^{a-1} - \frac{1}{2}\rceil} \kappa^b x^{2\lceil 2^{b-1} - \frac{1}{2}\rceil} m = \kappa^{a+b} x^{2^{a+b} + 2^b} m = x \kappa^{a+b} x^{2^b}m\] with $a, b \geq 1$. Thus $\mathcal{C}_+^2 M \subseteq (x)$.

Let us now show that $\sigma(M, f^{\frac{1}{2} - \eps}) = \mathbb{F}_2[x]$. We have \[  \mathcal{C}_+^e \mathbb{F}_2[x] \subseteq \mathcal{C}_+^e M = \mathcal{C}_+^{e-1} \mathcal{C}_+ M \subseteq \mathcal{C}_+^{e-1} \mathbb{F}_2[x]\] and thus $\sigma(M, f^{\frac{1}{2} - \eps}) =\sigma(\mathbb{F}_2[x], f^{\frac{1}{2} - \eps})$. We claim that $\mathcal{C}_+ \mathbb{F}_2[x] = \mathbb{F}_2[x]$. For this we need to consider the exponent of $x$ in \[\kappa^e x^{2 \lceil (\frac{1}{2} - \eps)(2^e -1) \rceil}.\] This exponent is given by $\lceil 2^{e -1} - \frac{1}{2} - \eps 2^e + \eps\rceil 2 \leq 2^e -1$ for $e \gg 0$. Thus $\mathcal{C}_+ \mathbb{F}_2[x] = \mathbb{F}_2[x]$.
\end{Bsp}

In \cite[Proposition 4.10 (4)]{fujinotakagischwedenonlc} the authors prove that $\sigma(R, f^\lambda) \subseteq \tau(R, f^{\lambda - \eps})$ if the non-$F$-regular locus of $R$ is contained in $V(f)$ (i.e.\ $R_f$ is $F$-regular). It is also observed in \cite[Remark 14.11]{fujinotakagischwedenonlc} that, contrary to the situation in characteristic zero, one does not have $\sigma(R, f^\lambda) = \tau(R, f^{\lambda - \eps})$ even if $R$ is ($F$-)regular.

We will in fact see that these properties extend to the module situation and also show that this "defect" of the non-$F$-pure filtration is constrained to those $\lambda$ for which the denominator is divisible by $p$.

Recall that, according to our running Assumption \ref{assumption}, $f$ is an $M$-regular element and thus not contained in any associated prime of $M$. Thus, $M_f$ is $F$-regular if and only if the constant sequence $(f, \ldots, f)$ is a sequence of test elements for $M$. Equivalently this means that $V(f)$ contains the non-$F$-regular locus of $M$.

\begin{Le}
\label{TauTestelementManipulation}
Let $(M, \kappa)$ be a Cartier module and $f$ an $M$-regular element. Assume that $\tau(M, f^\lambda)$ exists. If $M_f$ is $F$-regular, then \[\tau(M, f^\lambda) = \sum_{e \geq e_0} \mathcal{C}_e f^s M\] for all $s \gg 0$ and any $e_0 \geq 0$ where $\mathcal{C}$ is either of the two Cartier algebras
\[\bigoplus_{e \geq 0}\kappa^e f^{\lceil \lambda (p^e -1)\rceil} R \text{ or } \bigoplus_{e \geq 1} \kappa^e f^{\lceil \lambda (p^e)\rceil} R \oplus R.\]
\end{Le}
\begin{proof}
By $F$-regularity of $M_f$ the inclusion $\tau(M, f^{\lambda})_f \subseteq M_f$ is an equality. Thus we find $s \gg 0$ such that $f^s M \subseteq \tau(M, f^{\lambda})$. Set $N \coloneqq \sum_{e \geq e_0} \mathcal{C}_e f^s M$ and note that it is a $\mathcal{C}$-module. Thus we get an inclusion $N \subseteq \tau(M, f^\lambda)$ and, after inverting $f$, we obtain $N_f = \sum_{e \geq e_0} \kappa^e M_f = M_f$ since $M_f$ is $F$-pure with respect to $\kappa$. In particular, $H^0_\eta(N)_\eta = H^0_\eta(M)_\eta$ for all associated primes $\eta$ of $M$ since $f \notin \eta$. Since $\tau(M, f^\lambda)$ is minimal with this property the assertion follows.
\end{proof}

\begin{Prop}
\label{SigmaEqualsTauFregularLocus}
Assume that $R$ is essentially of finite type over an $F$-finite field. Let $(M, \kappa)$ be a Cartier module and $\lambda$ a positive rational number.
\begin{enumerate}[(a)]
\item{If $p$ does not divide the denominator of $\lambda$, then $\sigma(M, f^\lambda) \supseteq \tau(M, f^{\lambda - \eps})$ for any $0 < \eps \ll 1$.}
\item{If $M_f$ is $F$-regular, then $\sigma(M, f^\lambda) \subseteq \tau(M, f^{\lambda - \eps})$ for any $\lambda > 0$ and any $0 < \eps \ll 1$.}
\end{enumerate}
\end{Prop}
\begin{proof} We start by proving (a). By definition $\sigma(M, f^\lambda) = \mathcal{C}_+^a M$ for all $a \gg 0$. Fix $e$ such that $\lambda(p^e -1)$ is an integer. Then
\[\mathcal{C}_+^a M \supseteq \mathcal{C}_+^a \tau(M, f^0) \supseteq (\mathcal{C}_e)^a \tau(M, f^0) = \kappa^{ae} f^{\lambda(p^{ae} -1)} \tau(M, f^0) = \tau(M, f^{\lambda - \frac{1}{p^{ea}}}),\]  where we use Skoda for test modules and Lemma \ref{CartierTauDivisionByP}. 

Part (b) is an adaptation of \cite[Proposition 4.10 (4)]{fujinotakagischwedenonlc} to modules. Fix $0 < \eps \ll 1$.
Using Lemma \ref{TauTestelementManipulation} we find $s \gg 0$ such that \begin{equation}\label{Eq.TauComputation} \tau(M, f^{\lambda - \eps}) = \sum_{e \geq e_0} \kappa^e f^{\lceil(\lambda - \eps)p^e\rceil} f^s M\end{equation} for any $e_0 \geq 0$.

Choose $h$ such that $\sigma(M, f^\lambda) = \mathcal{C}_+^h M$ and such that $\lceil \lambda (p^e -1)\rceil \geq \lceil (\lambda - \eps) p^e \rceil + s$ for any $e \geq h$. Then

\begin{align*}\sigma(M, f^\lambda) = \mathcal{C}_+^h M\subseteq \mathcal{C}_{\geq h} M = \sum_{e \geq h} \kappa^e f^{\lceil \lambda (p^e -1) \rceil} M \subseteq \sum_{e \geq h} \kappa^e f^{\lceil (\lambda - \eps) p^e \rceil} f^s M  = \tau(M, f^{\lambda -\eps})\end{align*} where we use (\ref{Eq.TauComputation}) for the last equality.
\end{proof}

For the next result we need both existence of $\tau$ as well as discreteness.

\begin{Ko}
Let $R$ be of finite type over an $F$-finite field. Let $(M, \kappa)$ be a Cartier module and $\lambda$ a positive integer. Assume that $M_f$ is $F$-regular. If the denominator of $\lambda$ is not divisible by $p$, then $\sigma(M,f^\lambda) = \sigma(M, f^{\lambda - \eps})$ for all $0 < \eps \ll 1$.
\end{Ko}
\begin{proof}
By Lemma \ref{SigmaEqualsTauFregularLocus} there is $0 < \eps \ll 1$ such that $\sigma(M, f^\lambda) = \tau(M, f^{\lambda - \eps})$. Choosing $\eps$ suitably we may assume that the denominator of $\lambda - \eps$ is also not divisible by $p$. Thus we find $0 < \eps' \ll 1$ such that $\sigma(M, f^{\lambda - \eps}) = \tau(M, f^{\lambda - \eps -\eps'})$. By discreteness of $\tau$ we conclude that $\sigma(M, f^{\lambda - \eps}) = \sigma(M, f^\lambda)$.
\end{proof}

\begin{Bsp}
We do not get any a similar connection between $\tau(M,f^\lambda)$ and $\sigma(M, f^\lambda)$ if $\lambda = 0$. Consider for example $M = \mathbb{F}_p[x] \cdot x^{-1}$ with the usual Cartier structure $\kappa$. Then $(M, \kappa)$ is $F$-pure while $\tau(M) = \mathbb{F}_p[x]$.

The connection between the filtrations $\sigma$ and $\tau$ does no longer hold without the assumption that $M_f$ is $F$-regular. For instance, consider $\mathbb{F}_p[x,y] \cdot y^{-1}$  with the Cartier structure induced from $\mathbb{F}_p[x,y]$ and let $f = x$. Then \[\tau(M, x^{\lambda}) = (x^{\lfloor \lambda \rfloor}) \text{ and } \sigma(M, x^{\lambda}) = \mathbb{F}_p[x,y] \cdot y^{-1}\] for any $0 < \lambda \leq 1$. Moreover, $\sigma(M, x^{1 +\eps}) = \mathbb{F}_p[x,y] \cdot xy^{-1}$ for any $0 < \eps \ll 1$.

Indeed, \[\tau(M, x^\lambda) = \tau(\tau(M, x^0), x^{\lambda}) = \tau(\mathbb{F}_p[x,y], x^{\lambda}) = x^{\lfloor \lambda \rfloor},\] since $\mathbb{F}_p[x,y] \subseteq \mathbb{F}_p[x,y] \cdot y^{-1}$ agree at all associated($=$minimal) primes and $\mathbb{F}_p[x,y]$ is $F$-regular.

For the other equality, note that \begin{align*} \mathcal{C}_+ \mathbb{F}_p[x,y] \cdot y^{-1} &= \sum_{e \geq 1} \kappa^e x^{\lceil \lambda (p^e -1) \rceil } \frac{y^{p^{e-1}}}{y^{p^e}} \mathbb{F}_p[x,y] = y^{-1} \sum_{e \geq 1} \kappa^e y^{p^e -1} x^{\lceil \lambda(p^e -1) \rceil} \mathbb{F}_p[x,y]\end{align*}
From this computation we see that, if $\lambda \leq 1$, then $\mathcal{C}_+ \mathbb{F}_p[x,y] \cdot y^{-1} = \mathbb{F}_p[x,y] \cdot y^{-1}$. Thus applying $\mathcal{C}_+^e$ for $e \gg 0$ shows that $\sigma(M, f^\lambda) = \mathbb{F}_p[x,y] \cdot y^{-1}$ for $0 \leq \lambda \leq 1$.

If $\lambda = 1 +\eps$, then the inclusion $\mathcal{C}_+^h \subseteq \mathcal{C}_{\geq h}$ shows via the above computation that $\sigma(M, f^{\lambda}) \subseteq \mathbb{F}_p[x,y] \cdot xy^{-1}$. Since the filtration $\sigma$ is non-increasing we only have to show the converse inclusion for $\eps$ with denominator not divisible by $p$. In this case, fix $a$ such that $(1+\eps)(p^a -1)$ is an integer. Then
\[(\kappa^a x^{(1+\eps)(p^a -1)})^h x^{p^{ah} -\eps (p^{ah} -1)} y^{-1}  = \kappa^{ah} x^{2p^{ah} -1} \frac{y^{p^{ah} -1}}{y^{p^{ah}}}= y^{-1} x \kappa^{ah} y^{p^{ah} -1} x^{p^{ah} -1} = xy^{-1} \in \mathcal{C}_+^h M.\]
\end{Bsp}

In \cite{fujinotakagischwedenonlc} the authors introduce some other candidates for the non-$F$-pure ideal and ask whether they all coincide. In what follows we write $\mathcal{C}$ for the Cartier algebra generated in degree $e$ by $\kappa^e f^{\lceil \lambda (p^e -1)\rceil}$. The following variants are considered:
\begin{enumerate}[(a)]
\item{For a fixed $n \gg 0$ define $\sigma_n(M, f^\lambda) \coloneqq (\mathcal{C}_{\geq n})^h M$ for $h \gg 0$.}
\item{Suppose that $\lambda(p^e -1) \in \mathbb{Z}$ for some sufficiently large and sufficiently divisible $e$. Write $\mathcal{F}$ for the algebra that is given by $\bigoplus_{a \geq 0} \kappa^{ea} f^{\lambda (p^{ea} -1)} R$. Then set $\sigma'(M, f^\lambda) \coloneqq \mathcal{F}_+^h M$ for $h \gg 0$.}
\end{enumerate}

Once again note that by \cite[Proposition 2.13]{blicklep-etestideale} all these notions are well-defined for any $F$-finite ring $R$.

We point out that in \cite{fujinotakagischwedenonlc} non-$F$-pure ideals are also studied in the context of non-principal ideals (i.e. if $f$ is replaced by an ideal $\mathfrak{a} \subseteq R$). The authors then work with the integral closures of the $\mathfrak{a}^{\lceil \lambda(p^e -1) \rceil}$ making things more subtle (if $R$ is normal then any principal ideal generated by a non-zero-divisor is integrally closed -- see e.g.\ \cite[Proposition 10.2.3]{brunsherzog}).

\begin{Prop}
\label{Replacinging4.9}
Let $R$ be $F$-finite and $(M, \kappa)$ a Cartier module. Then $\sigma_n(M, f^\lambda) = \sigma(M, f^\lambda)$ for all $n \gg 0$. Assume now that $R$ is essentially of finite type over an $F$-finite field. If $M_f$ is $F$-regular and the denominator of $\lambda$ is not divisible by $p$, then $\sigma(M,f^\lambda) = \sigma_n(M,f^\lambda) = \sigma'(M, f^\lambda)$ for all $n \gg 0$.
\end{Prop}
\begin{proof}
Note that one clearly has containments $\sigma'(M, f^\lambda) \subseteq \sigma_n(M, f^\lambda) \subseteq \sigma(M, f^\lambda)$ whenever these objects are defined.

We first show that $\sigma(M, f^\lambda) \subseteq \sigma_n(M, f^\lambda)$. Let $h \gg 0$ be such that $\sigma(M,f^\lambda) = \mathcal{C}_+^h M$ and $\sigma_n(M, f^\lambda) = (\mathcal{C}_{\geq n})^h M$. Then $\sigma(M, f^\lambda) = \mathcal{C}_+^{nh} M \subseteq (\mathcal{C}_{\geq n})^h M = \sigma_n(M, f^\lambda)$.

Assume now that the denominator of $\lambda$ is not divisible by $p$. Since $f$ is a test element, we know, by Lemma \ref{SigmaEqualsTauFregularLocus}, that $\sigma(M,f^\lambda) = \tau(M, f^{\lambda - \eps})$ for $0 < \eps \ll 1$. Therefore, it suffices to prove that \[\sigma'(M, f^\lambda) \supseteq \tau(M, f^{\lambda -\eps}).\] To see this let us denote by $\mathcal{F}$ the Cartier algebra $\bigoplus_{e \geq 0} \kappa^{ea} f^{\lambda(p^{ea} -1)} R$. The argument now proceeds as in Lemma \ref{SigmaEqualsTauFregularLocus}.

To wit \[\mathcal{F}_+^h M \supseteq \mathcal{F}_+^h \tau(M, f^0) \supseteq (\mathcal{F}_{1})^h \tau(M, f^0) = \kappa^{ah} f^{\lambda(p^{ah} -1)} \tau(M, f^0) = \tau(M, f^{\lambda - \frac{1}{p^{ah}}}).\]
\end{proof}

\begin{Ko}
\label{SimpleSigmaDescription}
Let $R$ be essentially of finite type over an $F$-finite field and $(M, \kappa)$ a Cartier module. Assume that $M_f$ is $F$-regular. If the denominator of $\lambda$ is not divisible by $p$, then $\sigma(M, f^\lambda) = \kappa^e f^{\lambda(p^e-1)} M$ for all $e \gg0$ such that $\lambda(p^e -1)$ is an integer.
\end{Ko}
\begin{proof}
By Proposition \ref{Replacinging4.9} we have $\sigma(M, f^\lambda) = \sigma'(M, f^\lambda)$. Let $\mathcal{E}$ be the Cartier algebra $\bigoplus_{e \geq 0} \kappa^{ea} f^{\lambda(p^{ea} -1)} R$, where $\lambda(p^a -1)$ is an integer.

Note that \[(\mathcal{E}_+)^2 = \bigoplus_{e + f \geq 2} \kappa^{(e+f)a} f^{\lambda(p^{(e+f)a -1)}} = \mathcal{E}_{e \geq 2}\] and thus inductively $\mathcal{E}_+^e = \mathcal{E}_{\geq e}$. Since $\kappa^{ea}(f^{\lambda(p^{ea} -1)}M) \subseteq M$ we have \[\kappa^{2ea} f^{\lambda (p^{2ea} -1)} M  = \kappa^{ea} f^{\lambda (p^{ea} -1)} \kappa^{ea} f^{\lambda(p^{ea} -1)} M \subseteq \kappa^{ea} f^{\lambda(p^{ea} -1)} M.\] Thus $\sigma'(M,f^\lambda) = \mathcal{E}_{\geq e} M = \mathcal{E}_e M = \kappa^{ea} f^{\lambda(p^{ea} -1)} M$ for all $e \gg 0$.
\end{proof}

\begin{Prop}
\label{PS}
Assume that $R$ is essentially of finite type over an $F$-finite field and $(M, \kappa)$ is a Cartier module. Furthermore, assume that $M_f$ is $F$-regular and that the denominator of $\lambda$ is divisible by $p$. Then $\sigma(M, f^\lambda) = \sigma(M, f^{\lambda + \eps})$ for all $0 < \eps \ll 1$. 
\end{Prop}
\begin{proof}
As the filtration is non-increasing we just need to deal with the inclusion from left to right.
The $R$-module $\mathcal{C}_+ M \subseteq M$ is finitely generated. Hence, we have $\mathcal{C}_+ M = \sum_{e=1}^{E_1} \mathcal{C}_e M$. Note that it is harmless to pass to a larger $E_1$ in doing this. Repeating the argument with $\mathcal{C}_+^2 M\subseteq \mathcal{C}_+ M$ we find that $\mathcal{C}_+^2 M = \sum_{e=1}^{E_2} \mathcal{C}_e \mathcal{C}_+ M$. Repeating this argument we conclude that $\mathcal{C}_+^{h+1} M = (\sum_{e=1}^{E} \mathcal{C}_e)^{h+1} M$, where $E$ is the maximum over the $E_1, \ldots, E_{h+1}$.

In particular, if we choose $h \gg 0$ such that $\sigma(M, f^\lambda) = \mathcal{C}_+^h M$, then
\[ ( \sum_{e=1}^E \mathcal{C}_e)^{h+1} M = \mathcal{C}_+ \mathcal{C}_+^h M = \mathcal{C}_+^h M = ( \sum_{e=1}^E \mathcal{C}_e)^h M. \] Iterating this we deduce that $\sigma(M, f^\lambda) = (\sum_{e=1}^E \mathcal{C}_e)^h M$ for some fixed large $E$ and all $h \gg 0$. By our assumption on $\lambda$ we have $\lceil \lambda (p^e -1) \rceil = \lambda(p^e - 1) + \delta_e$ with $\delta_e > 0$ for all $e =1, \ldots, E$.

We can write a homogeneous element of $(\sum_{e=1}^E \mathcal{C}_e)^h$ as \begin{equation}\label{Eq.TypicalElement}\kappa^{e_1} f^{\lceil \lambda (p^{e_1} -1) \rceil} \cdots \kappa^{e_h} f^{\lceil \lambda (p^{e_h} -1) \rceil} = \kappa^a f^{\lambda (p^a -1) + \sum_{i=1}^h \delta_{e_i} p^{\sum_{j=i+1}^h e_j}}, \end{equation} where $a = e_1 + \ldots + e_h$.

Let us now analyze the second term of the exponent of $f$ in (\ref{Eq.TypicalElement}) after dividing by $p^a -1$.
\begin{claim}
\label{claim.exponent}
The expression \[A \coloneqq \frac{\sum_{i=1}^h \delta_{e_i} p^{\sum_{j={i+1}}^h e_j}}{p^{e_1 + \ldots + e_h} -1}\] is bounded away from zero by a constant independent of $h$.
\end{claim}
\begin{proof}[proof of \ref{claim.exponent}]
Note that, as noted above, the $\delta_i$ are all contained in a finite set of positive numbers. Let $\delta$ be the minimum over this set, then
\[A \geq \delta \frac{\sum_{i=1}^h p^{\sum_{j=i+1}^h e_j}}{p^{e_1 + \ldots + e_h} -1} \geq \delta \frac{p^{e_2 + \ldots +e_h}}{p^{e_1 + \ldots + e_h} -1} \geq \delta \frac{p^{e_2 + \ldots +e_h}}{p^{e_1 + \ldots + e_h}}  = \delta p^{-e_1} \geq \delta p^{-E}. \]
\end{proof}

Taking for instance $\eps = \delta p^{-E}/2$ we then get $\mathcal{C}_+^h M = \sum_{a=h}^{Eh} \kappa^a f^{(\lambda + \eps)(p^a -1)} f^{r(a)} M$, where $r(a)$ grows with $a$. In particular, taking a larger $h$ we can arrange for all $r(a)$ to be larger than a given integer $s$. This allows us to apply Lemma  \ref{TauTestelementManipulation}, i.e.\ we get that \[\mathcal{C}_+^h M = \sum_{a=h}^{Eh} \kappa^a f^{(\lambda + \eps)(p^a -1) +r(a)} M \subseteq \sum_{a\geq h} \kappa^a f^{\lceil(\lambda + \eps)(p^a -1)\rceil} f^{s} M = \tau(M, f^{\lambda + \eps}).\]
Now we can use Lemma \ref{SigmaEqualsTauFregularLocus} to conclude that $\tau(M, f^{\lambda + \eps}) = \sigma(M, f^{\lambda + \eps + \eps'})$, where we choose $0 < \eps' \ll 1$ in such a way that the denominator of $\lambda + \eps + \eps'$ is not divisible by $p$.
\end{proof}

We have just seen in the proof above that the following holds:

\begin{Ko}
\label{Sigmataupdenominator}
Assume that $R$ is essentially of finite type over an $F$-finite field, that $(M, \kappa)$ is a Cartier module and that $\lambda$ is a positive rational number. Assume that $f$ is a test element for $M$. If the denominator of $\lambda$ is divisible by $p$, then $\sigma(M, f^\lambda) = \tau(M, f^{\lambda + \eps})$.
\end{Ko}

\begin{Bem}
\label{anderealgebra}
At this point one might wonder whether one obtains unconditional left/right-continuity if one instead considers $\mathcal{E}_+^h M$ where \[\mathcal{E} = \bigoplus_{e \geq 1} \kappa^e f^{\lceil \lambda p^e \rceil} R \oplus R\] (the last summand being $\mathcal{E}_0$). This is however not the case.

Indeed, if $(M,\kappa)$ is $F$-regular, then $\mathcal{E}_+^h M \subseteq \mathcal{E}_{\geq h} M = \tau(M, f^{\lambda})$, where the equality is due to Theorem \ref{TauFRegularDescription} and \cite[Lemma 4.1]{blicklestaeblerbernsteinsatocartier}. By minimality of $\tau$ we then obtain equalities everywhere. Hence, in this case the filtration is right-continuous but not left-continuous.

If $(M, \kappa)$ is not $F$-regular but $F$-pure and $M_f$ is $F$-regular, then for $\lambda = 0$ one has $\mathcal{E}_+^h M = M$ while  we claim that for any $0< \lambda \ll 1$ one has $\mathcal{E}_+^h M = \tau(M, f^0)$. Namely, \[\mathcal{E}_+^h M \subseteq \mathcal{E}_{\geq h} M = \sum_{e \geq h} \kappa^e f^{\lceil \lambda p^e \rceil} M \subseteq \sum_{e \geq h} \kappa^e f^{\lceil (\lambda - \eps) p^e \rceil} f^s M = \tau(M, f^{\lambda - \eps})\] where the last equality is due to Lemma \ref{TauTestelementManipulation}. By right-continuity of $\tau$ we have $\tau(M, f^0) = \tau(M, f^{\lambda - \eps}) =\tau(M, f^\lambda)$. By minimality of $\tau$ we conclude that $\mathcal{E}_+^h M = \tau(M, f^\lambda)$.
\end{Bem}

\begin{Que}
Is the filtration defined in Remark \ref{anderealgebra} right-continuous for $\lambda > 0$ for general $F$-pure $M$?
\end{Que}

If we filter along a non-principal ideal, then there are two natural generalizations that one can consider:

\begin{Def}
Fix an $F$-finite ring $R$, an ideal $\mathfrak{a} \subseteq R$ and a non-negative rational number $\lambda$. If $\mathcal{C}$ is a Cartier algebra and $M$ a $\mathcal{C}$-module, then we define the \emph{non-$F$-pure module (with respect to $\mathcal{C}, \mathfrak{a}$ and $\lambda$}) as \[ \sigma(M, \mathfrak{a}^\lambda) = \mathcal{E}_+^h M \quad(h \gg 0),\] where $\mathcal{E}$ is the Cartier algebra generated in degree $e$ by $\mathcal{C}_e \mathfrak{a}^{\lceil \lambda (p^e -1)\rceil}$.

We also consider a variant where we work with integral closures: \[\bar{\sigma}(M, \mathfrak{a}^\lambda) = \mathcal{F}_+^h M \quad (h \gg 0),\] where $\mathcal{F}$ is the Cartier algebra generated in degree $e$ by \[\mathcal{C}_e \overline{\mathfrak{a}^{\lceil\lambda(p^e -1)\rceil}.}\]
\end{Def}

As mentioned earlier the notion $\bar{\sigma}$ is the one considered in \cite{fujinotakagischwedenonlc}.

We do not know how to prove an analog of Skoda's theorem for non-$F$-pure ideals. We can only prove one direction:

\begin{Le}
Let $(M, \kappa)$ be a Cartier module, $\mathfrak{a}$ an ideal and $\lambda \geq 1$ a rational number. Then $\mathfrak{a} \, \, \sigma(M, \mathfrak{a}^{\lambda -1}) \subseteq \sigma(M, \mathfrak{a}^{\lambda})$. Likewise, $\mathfrak{a} \, \, \bar{\sigma}(M, \mathfrak{a}^{\lambda -1}) \subseteq \bar{\sigma}(M, \mathfrak{a}^{\lambda})$.
\end{Le}
\begin{proof}
By definition $\sigma(M, \mathfrak{a}^{\lambda -1}) = \mathcal{C}_+^h M$ for all $h \gg 0$, where $\mathcal{C}$ is the Cartier algebra that is generated in degree $e$ by $\kappa^e \mathfrak{a}^{\lceil (\lambda -1)(p^e -1)\rceil}$. Hence, we can write any element of $\mathcal{C}_+^h$ as a sum of elements $C \coloneqq \prod_{i=1}^h \kappa^{e_i} \mathfrak{a}^{\lceil (\lambda -1)(p^{e_i} -1) \rceil}$.

Note that one has an inclusion $\mathfrak{a}^{[p^e]} \cdot \mathfrak{a}^{\lceil (\lambda-1)(p^e -1)\rceil} \subseteq \mathfrak{a}^{\lceil \lambda  (p^e -1) +1\rceil}$ for all $e$. Hence, \begin{align*}\mathfrak{a} \, \, C &= \kappa^{e_1} \mathfrak{a}^{[p^{e_1}]} \mathfrak{a}^{\lceil (\lambda -1)(p^{e_1} -1)\rceil}\prod_{i=2}^h \kappa^{e_i} \mathfrak{a}^{\lceil (\lambda -1) (p^{e_i}- 1 )\rceil}\\ &\subseteq \kappa^{e_1} \mathfrak{a}^{\lceil \lambda(p^{e_1} - 1) \rceil} \mathfrak{a} \prod_{i =2}^h \kappa^{e_i} \mathfrak{a}^{\lceil (\lambda -1)(p^{e_i} -1)} \subseteq\cdots \subseteq \prod_{i=1}^h \kappa^{e_i} \mathfrak{a}^{\lceil \lambda (p^{e_i} -1) \rceil} \mathfrak{a}.\end{align*}
Since $\mathfrak{a}M \subseteq M$ the claim follows.

The inclusion for $\bar{\sigma}$ follows similarly using the fact that $\overline{I} \cdot \overline{J} \subseteq \overline{IJ}$ for ideals $I, J$.
\end{proof}

Note that if $\mathfrak{a} = (f)$ and $M_f$ is $F$-regular, then we do get an analog of Skoda provided that $\lambda > 1$. We are therefore led to the following

\begin{Que}
\begin{enumerate}[(a)]
\item{Does a full Skoda theorem hold for $\sigma(M, \mathfrak{a}^\lambda)$ if $M$ is $F$-regular or in general? That is, if $\mathfrak{a}$ is generated by $m$ elements is it true that $\mathfrak{a} \,\,\sigma(M, \mathfrak{a}^{\lambda -1}) = \sigma(M, \mathfrak{a}^{\lambda})$ for $\lambda > m$?}
\item{Does a full Skodda theorem hold if $\mathfrak{a} = (f)$ in general?}
\end{enumerate}
\end{Que}

We may also ask

\begin{Que}
Ist the filtration $\sigma(M, f^\lambda)$ left-continuous if the $F$-jumping nubmers of $\tau(M, f^{\lambda})$ all have a denominator not divisible by $p$?
\end{Que}

\begin{Que}
If $M$ is $F$-regular and the denominator of $\lambda$ is not divisible by $p$, then is it true that $\sigma(M, \mathfrak{a}^\lambda) = \tau(M, \mathfrak{a}^{\lambda - \eps})$ for all $0 < \eps \ll 1$?
\end{Que}

\section{Bernstein-Sato polynomials attached to Cartier modules without $F$-regularity}
\label{BspSigmaSection}
The goal of this section is to investigate whether the correspondence between zeros of the Bernstein-Sato polynomial (Definition \ref{def.bernsteinsato}) and $F$-jumping numbers can be extend beyond the case where $(M, \kappa)$ is $F$-regular.

We can prove a partial generalization under the assumption that $M_f$ is $F$-regular. One could phrase this result still entirely in the language of test modules but we will use the language of non-$F$-pure modules (the translation being given by  Proposition \ref{SigmaEqualsTauFregularLocus}). The reason for this is that the author believes that, using non-$F$-pure modules, this connection might hold in greater generality. Essentially, the question is what are the minimal assumptions for Corollary \ref{SimpleSigmaDescription} to hold.

\begin{Def}
Let $R$ be $F$-finite, $(M, \kappa)$ a Cartier module and $f$ and $M$-regular element. Then we define
\[Gr^\lambda_\sigma M \coloneqq \sigma(M, f^{\lambda})/\sigma(M, f^{\lambda + \eps}).\]
\end{Def}

\begin{Le}
Let $R$ be essentially of finite type over an $F$-finite field and let $(M, \kappa)$ be a Cartier module. Assume that $M_f$ is $F$-regular. Then $\kappa^a f^{\lceil \lambda (p^a -1) \rceil}$ defines a Cartier structure on $Gr^\lambda_\sigma M$. If $Gr^\lambda_\sigma M \neq 0$, then this Cartier structure is non-nilpotent if and only if $\lambda(p^a -1) \in \mathbb{Z}$.
\end{Le}
\begin{proof}
If $\lambda = 0$, then $Gr^0_\sigma M = \underline{M}/\tau(M, f^{\eps})$ by Proposition \ref{SigmaEqualsTauFregularLocus}. Moreover, $\kappa^a \tau(M, f^{\eps}) = \tau(M, f^{\eps/p^a}) = \tau(M, f^0) = \tau(M, f^\eps)$ by right -continuity of $\tau$ and Lemma \ref{CartierTauDivisionByP}. So $\kappa$ acts on the quotient. Moreover, $\kappa$ acts surjectively on $\underline{M}$. So if the quotient is non-zero (which is the case if and only if $M$ is not $F$-regular), then the action is not nilpotent.

If $\lambda > 0$, then $Gr^\lambda_\sigma M = \tau(M, f^{\lambda - \eps'})/\tau(M, f^{\lambda + \eps -\eps'})$  by Proposition \ref{SigmaEqualsTauFregularLocus}, where $0 < \eps < \eps' \ll 1$. If this  quotient is zero, then the statement of the lemma is trivially true. The quotient is non-zero if and only if $\lambda$ is an $F$-jumping number. In this case, the statement follows from Theorem \ref{TheoNilpotentGr}
\end{proof}

\begin{Theo}
\label{LastTheo}
Let $R$ be regular and essentially of finite type over an $F$-finite field. Let $(M, \kappa)$ be a Cartier module, $f \in R$ an $M$-regular element and assume that $M_f$ is $F$-regular. Let $\lambda \in (0,1] \cap \mathbb{Z}_{(p)}$. If $\frac{\lceil\lambda p^e\rceil -1}{p^e}$ is a zero of the $e$th Bernstein-Sato polynomial (see Definition \ref{def.bernsteinsato}) for some $e \gg 0$ such that $\lambda(p^e -1) \in \mathbb{Z}$, then $Gr_\sigma^\lambda M$ is non-trivial.
\end{Theo}
\begin{proof}
Since $\lambda(p^e-1) = \lambda p^e - \lambda \in \mathbb{Z}$ we have $\lceil \lambda p^e\rceil -1 = \lambda(p^e -1)$. We denote the $p$-adic expansion of $\lambda$ by $\sum_{i\geq1} c_i p^{-i}$. Then our assumption that $\frac{\lceil\lambda p^e\rceil -1}{p^e}$ is a zero of the Bernstein-Sato polynomial means that the generalized $(c_e,\ldots, c_1)$-eigenspace of $N_e$ (as defined just before Theorem \ref{BSPGeneral1}) is non-trivial. Using that ${F^e}^!$ is fully faithful this is equivalent to
\begin{equation}\label{eq.??}\kappa^e f^{c_e + c_{e-1}p +\ldots + c_1 p^{e-1}} M \neq \kappa^e f^{1 + c_e + c_{e-1}p +\ldots + c_1 p^{e-1}} M \end{equation} (cf.\ the proof of Theorem \ref{BSPGeneral1}). By our assumption that $\lambda(p^e-1)$ the sequence $c_i$ is periodic with period length dividing $e$. Hence, (\ref{eq.??}) is equivalent to 
\[\kappa^e f^{c_1 + c_{2}p +\ldots + c_e p^{e-1}} M \neq \kappa^e f^{1 + c_1 + c_{2}p +\ldots + c_e p^{e-1}} M \]
Now the left-hand side of this inequality is equal to $\kappa^e f^{\lambda(p^e -1)} M = \sigma(M, f^\lambda)$ by Corollary \ref{SimpleSigmaDescription} (where we need that $e \gg 0$). Fix $0 < \eps \ll 1$ such that $(\lambda+ \eps)(p^e -1) \in \mathbb{Z}$ (this works for suitable $e \gg 0$). Using Corollary \ref{SimpleSigmaDescription} once more, we then have \[\sigma(M, f^{(\lambda + \eps)}) = \kappa^e f^{(\lambda + \eps)(p^e -1)} M \subseteq \kappa^e f^{1 + \lambda(p^e -1)} M \neq \kappa^e f^{\lambda(p^e -1)} M = \sigma(M, f^\lambda).\] Hence, $Gr_\sigma^\lambda M \neq 0$.
\end{proof}

\begin{Bsp}
Let us illustrate that the other direction of Theorem \ref{LastTheo} does not hold. That is, it may happen that $\sigma(M, f^\lambda) \neq \sigma(M, f^{\lambda + \eps})$ but $\frac{\lceil \lambda p^e \rceil -1}{p^e}$ is no a zero of any $e$th Bernstein-Sato polynomial.

Consider $M = \mathbb{F}_p[x] \cdot x^{-1}$, $f =x^2$ and $\lambda =\frac{1}{2}$ where we assume $p \neq 2$. Then $\sigma(M, f^\lambda) \neq \sigma(M, f^{\lambda + \eps}$). Indeed, Corollary \ref{SimpleSigmaDescription} yields that $\sigma(M, f^{\lambda}) = \kappa^e x^{p^e -1} x^{-1}M = \mathbb{F}_p[x]$ for $e \gg 0$. Similarly $\sigma(M, f^{\lambda + \eps}) = \kappa^e x^{p^e -1} x^{2\eps(p^e-1)} x^{-1} M$ if we choose $0 < \eps \ll 1$ so that $\lambda + \eps$ has denominator not divisible by $p $ and $e \gg 0$ suitable. In particular, we may choose $e$ so large that $2 \eps (p^e -1) \geq 2$, then we see that $\sigma(M, f^{\lambda + \eps}) \subseteq (x)$.

As in the proof of the theorem $\frac{\lceil \lambda p^e \rceil -1}{p^e}$ is a zero of any $e$th Bernstein-Sato polynomial if and only if
\[\kappa^e f^{\lambda(p^e -1)} x^{-1} \mathbb{F}_p[x] \neq \kappa^e f^{\lambda(p^e -1) +1} x^{-1} \mathbb{F}_p[x]\]
Now the left-hand side evaluates to $\kappa^e x^{p^e-1} x^{-1} \mathbb{F}_p[x] = \mathbb{F}_p[x]$ for any $e \geq 1$. The right-hand side similarly evaluates to $\kappa^e x^{p^e -1} \mathbb{F}_p[x] = \mathbb{F}_p[x]$.
\end{Bsp}

\begin{Bem}
\begin{enumerate}[(a)]
\item{If $(M, \kappa)$ is $F$-regular then all non-nilpotent information is recovered. Indeed, for suitable $e \gg 0$ one has $\sigma(M, f^{\lambda+\eps}) = \kappa^e f^{(\lambda + \eps)(p^e -1)} M = \tau(M, f^{\lambda + \eps - \frac{\lambda + \eps}{p^e}})$, and $\kappa^e f^{\lambda(p^e -1)} f M = \tau(M, f^{\lambda + \frac{1 - \lambda}{p^e}})$ e.g.\ using Theorem \ref{TauFRegularDescription} and Corollary \ref{SimpleSigmaDescription}. By right-continuity of $\tau$ these two quantities coincide for $0 < \eps \ll 1$ and $e \gg 0$ suitable.}
\item{Note that in characteristic zero one has that if $\lambda$ is a jumping number of the multiplier ideal filtration then it is a zero of the Bernstein-Sato polynomial. Here the situation is the other way around.}
\end{enumerate}
\end{Bem}
\bibliography{bibliothek.bib}
\bibliographystyle{amsalpha}
\end{document}